\documentclass[10pt,draft,twoside]{article}

\usepackage{etex}                                               
\usepackage[english]{babel}                                     
\usepackage{amsmath,amsthm,amssymb,amsfonts,latexsym,cancel}    
\usepackage{enumerate}                                          
\usepackage{enumitem} \setlist[itemize]{noitemsep,topsep=0pt}   
\usepackage{pifont}                                             
\usepackage[latin1]{inputenc}                                   
\usepackage{lmodern}                                            
\usepackage[T1]{fontenc}                                        
\usepackage[none]{hyphenat}                                     
\usepackage{ragged2e} \justifying                               
\usepackage{marvosym,manfnt,calligra}                           
\usepackage[mathscr]{euscript}                                  
\usepackage{makeidx} \makeindex                                 
\usepackage[dvips]{graphicx}                                    
\usepackage{graphicx}                                           
\usepackage{color}                                              
\usepackage{verbatim}                                           
\usepackage{amssymb}                                            
\usepackage{amsfonts}                                           
\usepackage{amsmath}                                            
\usepackage{latexsym}                                           
\usepackage{amsthm}                                             
\usepackage{dsfont}                                             
\usepackage{float}                                              
\usepackage{wrapfig}                                            
\usepackage[rflt]{floatflt}                                     
\usepackage{nccmath}                                            
\usepackage{tikz}                                               
\usetikzlibrary{intersections}                                  
\usetikzlibrary{mindmap,backgrounds,shapes,arrows,chains}       
\usepackage[all,matrix,arrow]{xy}                               
\usepackage{pgf,tikz} \usetikzlibrary{arrows}                   
		
\newtheoremstyle{newstyle}   
{} 			 				 
{12pt plus 1pt} 			 
{\mdseries} 				 
{} 
{\bfseries} 				 
{.} 						 
{ } 						 
{} 							 
		
\theoremstyle{newstyle}
		
\theoremstyle{newstyle}

\newtheorem{definition}{Definition}[section]
\newtheorem{theorem}[definition]{Theorem}
\newtheorem{corollary}[definition]{Corollary}
\newtheorem{lemma}[definition]{Lemma}
\newtheorem{proposition}[definition]{Proposition}

\allowdisplaybreaks
\DeclareMathAlphabet{\mathcalligra}{T1}{calligra}{m}{n}
\DeclareMathAlphabet{\mathpzc}{OT1}{pzc}{m}{it}

\title{\textbf{Topological games in Ramsey spaces}}
\author{ Juli\'an C. Cano $\;\;\;$ and $\;\;\;$ Carlos A. Di Prisco }
\date{}

\newcommand{\Addresses}{{
		\bigskip
		\footnotesize
		
		Juli\'{a}n C. Cano, \textsc{Universidad de Los Andes (Bogot\'a).}\par\nopagebreak
		\textit{E-mail address}, J.C. Cano: \texttt{jc.canor@uniandes.edu.co}
		
		\medskip
		
		Carlos A. Di Prisco, \textsc{Universidad de Los Andes (Bogot\'a),  Instituto Venezolano de Investigaciones Cient\'{\i}ficas (Caracas), and Universidad Nebrija (Madrid).}\par\nopagebreak
		\textit{E-mail address}, C.A. Di Prisco: \texttt{ca.di@uniandes.edu.co}
		
}}

\begin{document}

\maketitle

\sloppy

\kern1em
\begin{abstract}
\noindent Topological Ramsey theory studies a class of combinatorial topological spaces, known as topological Ramsey spaces, unifying the essential features of those combinatorial frames where the Ramsey property is equivalent to the Baire property. In this article, we present a general overview of the combinatorial structure of topological Ramsey spaces and their main properties, and we propose an alternative proof of the abstract Ellentuck theorem for a large family of axiomatized topological Ramsey spaces. Additionally, we introduce the notion of selective axiomatized topological Ramsey space, and generalize Kastanas games in order to characterize the Ramsey property for this broad family of topological Ramsey spaces through topological games. 

\medskip


\medskip

\end{abstract}

\section{Introduction}

Infinite dimensional Ramsey theory has developed in different directions since its beginning in the 1960's. For instance, topological Ramsey spaces, first introduced by Carlson and Simpson in \cite{Carlson-Simpson2}, and later reformulated and extended by Todor\v{c}evi\'{c} in \cite{Todorcevic(BookRamsey)}, provide a general framework to study a wide variety of combinatorial properties that have topological counterparts.

\medskip

The prototype example of the topological characterization of an infinite dimensional combinatorial property is Ellentuck's result in \cite{Ellentuck}, which establishes an infinite dimensional version of Ramsey's theorem restricted to Baire measurable partitions of $[\mathbb{N}]^{\omega}$. On the Ramsey property of subsets of $[\mathbb{N}]^{\omega}$, Erd\H{o}s and Rado proved in \cite{Erdos-Rado}, using the axiom of choice, that there exist subsets of $[\mathbb{N}]^{\omega}$ which are not Ramsey. Later, Galvin and Prikry showed in \cite{Galvin-Prikry} that Borel subsets of $[\mathbb{N}]^{\omega}$ in the metrizable topology are Ramsey, and Silver in \cite{Silver} extended this fact to analytic subsets. Finally, Ellentuck in \cite{Ellentuck} characterized topologically Ramsey subsets of $[\mathbb{N}]^{\omega}$ as those having the Baire property with respect to the exponential topology. Ellentuck's proof of this result uses a powerful combinatorial technique, due to Nash-Williams in \cite{NashWilliams}, now known as \textit{combinatorial forcing} (see \cite{Todorcevic(BookRamsey), Todorcevic(BookTopology), Halbeisen}). However, Matet in \cite{Matet2} gives a shorter proof of Ellentuck's result avoiding the use of combinatorial forcing.

\medskip

Ellentuck's theorem can be generalized using the combinatorial properties of two types of families of infinite subsets of integers, namely selective and semiselective coideals. These kinds of coideals make it possible to construct local versions of Ellentuck's space, where the homogeneity of Borel partitions is delimited and located within a fixed coideal. The interaction between Ellenctuk's theorem and coideals on $\mathbb{N}$ has been studied by Mathias, Matet, and Todor\v{c}evi\'{c} in \cite{Mathias, Matet1,  Todorcevic(BookTopology)} for the notion of selectivity, and extended by Farah and Todor\v{c}evi\'{c} in \cite{Farah, Todorcevic(BookRamsey)} for the notion of semiselectivity.

\medskip

Kastanas in \cite{Kastanas} gave a characterization of the Ramsey property of subsets of $[\mathbb{N}]^{\omega}$ in terms of determined topological games. For a fixed partition of $[\mathbb{N}]^{\omega}$ into finite pieces, Kastanas presents an infinite game such that either player has a winning strategy if and only if some homogeneous set for the partition can be constructed. Local versions of Kastanas games have been proposed by Matet in \cite{Matet1} for the class of selective coideals, and extended by Di Prisco, Mijares, and Uzc\'{a}tegui in \cite{DiPrisco-Mijares-Uzcategui} for the class of semiselective coideals.

\medskip

Topological Ramsey spaces deal with the optimal conditions that must be imposed to a combinatorial structure in order to guarantee the existence of homogeneous substructures with respect to finite partitions whose parts are topologically definable. In this regard, topological Ramsey spaces form a general combinatorial framework in which the respective notions of Ramsey property and Baire property are equivalent.

\medskip

The Ellentuck space $[\mathbb{N}]^{\omega}$ (\cite{Ellentuck}) is the quintessential example of a topological Ramsey space, although there are several combinatorial contexts that correspond to specific examples of topological Ramsey spaces, such as: the space of infinite block sequences of finite subsets of $\mathbb{N}$ (\cite{Milliken1}); the space of strong subtrees of a fixed rooted pruned finitely branching tree of height $\omega$ (\cite{Milliken3}); the space of infinite partitions of $\mathbb{N}$ (\cite{Carlson-Simpson1}); the space of reduced echelon infinite matrices over a fixed finite field (\cite{Carlson1}); the space of rapidly increasing infinite sequences of variable words over a fixed alphabet (\cite{Carlson2}); the space of infinite block sequences of vectors of finite subsets of $\mathbb{N}$ (\cite{Gowers1}); the space of infinite block sequences of vectors of located variable words over a fixed alphabet (\cite{Todorcevic(BookRamsey)}); the hierarchy of Ramsey spaces $\{ \mathcal{R}_{\alpha} \}_{\alpha<\omega_{1}}$ corresponding to recursive projections of trees (\cite{Dobrinen-Todorcevic1,Dobrinen-Todorcevic2}); the hierarchy of Ramsey spaces $\{ \mathcal{E}_{k} \}_{k<\omega}$ corresponding to high dimensional Ellentuck spaces (\cite{Dobrinen1}); construction of topological Ramsey spaces through uniform barriers on $\mathbb{N}$ (\cite{Dobrinen2}); construction of topological Ramsey spaces from Fra\"{i}ss\'{e} classes (\cite{Dobrinen-Mijares-Trujillo}); spaces of alternating equivalence relations for finite partitions of $\mathbb{N}$ (\cite{Kawach-Todorcevic}); among others.

\medskip 

The main objective of this article is to propose an alternative proof of the abstract Ellentuck theorem, which concerns the Ramsey property for subsets of a topological Ramsey space as established in \cite{Todorcevic(BookRamsey)}, and to characterize this combinatorial property in terms of infinite games through an abstract version of Kastanas' theorem.

\medskip

The article begins with a quick review of the Ramsey property for subsets of the Ellentuck space of infinite subsets of $\mathbb N$ and the Kastanas game characterization (Section \ref{Ellentuckspace}). Then follows a presentation of Todor\v{c}evi\'{c}'s formulation of topological Ramsey spaces (Section \ref{TRS}). Later, we define strong axiomatized topological Ramsey spaces, and give an alternative proof of the abstract Ellentuck theorem for this class of Ramsey spaces that include a wide family of topological Ramsey spaces (Section \ref{alternativeproof}). Finally, we present an abstract notion of selectivity, and generalize the Kastanas game-theoretic characterization of the Ramsey property to selective axiomatized topological Ramsey spaces (Section \ref{abstractKastanas}).

\medskip
 
At this point, let us recall some important concepts related to any topological space $\mathcal{R}$. First, a set $\mathcal{N} \subseteq \mathcal{R}$ is \textit{nowhere dense} if for each basic open set $\mathcal{O}$ there is a basic open subset $\mathcal{Q} \subseteq \mathcal{O}$ such that $\mathcal{N} \cap \mathcal{Q} = \emptyset$. Also, a set $\mathcal{M} \subseteq \mathcal{R}$ is \textit{meager} if $\mathcal{M} = \bigcup_{n\in \mathbb{N}} \mathcal{N}_{n}$ where each set $\mathcal{N}_{n}$ is nowhere dense. In addition, a set $\mathcal{A} \subseteq \mathcal{R}$ has the \textit{abstract Baire property} if for every basic open set $\mathcal{O}$ there exists some basic open subset $\mathcal{Q} \subseteq \mathcal{O}$ such that either $\mathcal{Q} \subseteq \mathcal{A}$ or $\mathcal{A} \cap \mathcal{Q} = \emptyset$. Finally, a set $\mathcal{B} \subseteq \mathcal{R}$ has the \textit{Baire property} if the set $\mathcal{B} \triangle \mathcal{O}$ is meager for some open set $\mathcal{O}$.

\section{Ramsey property on Ellentuck space} \label{Ellentuckspace}

In this section, we briefly present the Ramsey property for families of infinite subsets of integers, and expose both \textit{Ellentuck theorem} and \textit{Kastanas theorem}. To do this, let us first fix some notation:

\medskip

Given any countably infinite set $X$ and $n\in \mathbb{N}$, we denote by $[X]^{n}$ the set of subsets of $X$ with exactly $n$ elements; also, we denote by $[X]^{<\omega}$ the set of finite subsets of $X$, hence $[X]^{<\omega} = \bigcup_{n\in \mathbb{N}} [X]^{n}$; and analogously, we denote by $[X]^{\omega}$ the set of infinite subsets of $X$.

\medskip

Given $s \in [\mathbb{N}]^{<\omega}$ and $A\in [\mathbb{N}]^{\omega}$, the notation $s \sqsubset A$ means that $s$ is initial segment of $A$. We denote by $A/s$ the set $A/s = \{ n\in A \, | \, s<n \}$, where we write $s<n$ whenever $s=\emptyset$ or $\max s <n$. Lastly, for each $s \in [\mathbb{N}]^{<\omega}$ and each $B\in [\mathbb{N}]^{\omega}$ we define the set $[s,B]$ as follows:
\begin{center}
	$[s,B] = \{ A\in [\mathbb{N}]^{\omega} \,|\, A \subseteq B \wedge s \sqsubset A\}$. 
\end{center}
The collection $\left \{ [s,B] \; | \; s\in [\mathbb{N}]^{<\omega} \wedge B\in [\mathbb{N}]^{\omega} \right \}$ forms an uncountable basis of closed-open sets for a non-metrizable topology over $[\mathbb{N}]^{\omega}$, called the \textit{Ellentuck topology} of $[\mathbb{N}]^{\omega}$. It should be noted that the Ellentuck topology is finer than the metrizable topology, which is obtained by identifying $[\mathbb{N}]^{\omega}$ as a subspace of the Cantor space $2^{\mathbb{N}}$.

\medskip

\textit{Ramsey theorem} is a deep generalization of the combinatorial result known as the \textit{pigeonhole principle}, which states that if any infinite set is partitioned  into a finite number of pieces, then at least one of these pieces is infinite.

\smallskip

\begin{theorem} \label{RamseyTheorem}
	(Ramsey, \cite{Ramsey}). For all $n,k \in \mathbb{N}$ with $n,k\neq 0$, if $[\mathbb{N}]^{n} = \bigcup_{i=1}^{k} G_{i}$ is any finite partition of $[\mathbb{N}]^{n}$, then there is $H\in [\mathbb{N}]^{\omega}$ such that $[H]^{n} \subseteq G_{i}$ for some $i\in\{1,\ldots,k\}$.
\end{theorem}

\smallskip

A set $\mathcal{A} \subseteq [\mathbb{N}]^{\omega}$ determines a partition of $[\mathbb{N}]^{\omega}$ into two parts, namely $\mathcal{A}$ and $\mathcal{A}^{\complement}$; so, one can ask when there exists some $H\in [\mathbb{N}]^{\omega}$ such that $[H]^{\omega}$ is contained in one of these parts. This question motivated the definition of the Ramsey property, which is a combinatorial notion concerning subsets of $[\mathbb{N}]^{\omega}$ that captures the optimal conditions to guarantee a suitable infinite dimensional version of Ramsey theorem.

\smallskip

\begin{definition} \label{RamseyProperty}
	A set $\mathcal{A} \subseteq [\mathbb{N}]^{\omega}$ is \textit{Ramsey} if for all $N \in [\mathbb{N}]^{\omega}$ and $s\in [N]^{<\omega}$, there is some $M\in [N]^{\omega}$ such that $s\in [M]^{<\omega}$ and either $[s,M] \subseteq \mathcal{A}$ or $[s,M] \cap \mathcal{A} = \emptyset$. Moreover, if the second option always holds, then we say that $\mathcal{A}$ is \textit{Ramsey null}.
\end{definition}

\smallskip

The Ramsey property has two important set-theoretic characterizations. The first one is given by the Ellentuck theorem, which states a topological characterization of the Ramsey property through the Baire property with respect to the Ellentuck topology on $[\mathbb{N}]^{\omega}$. The second one is given by the Kastanas theorem, which states a game-theoretic characterization of the Ramsey property through topological games on $[\mathbb{N}]^{\omega}$.

\smallskip
 
\begin{theorem} \label{EllentuckTheorem}
	(Ellentuck, \cite{Ellentuck}). Consider the space $[\mathbb{N}]^{\omega}$ endowed with the Ellentuck topology.
	\setlist{nolistsep}
	\begin{enumerate}
	\setlength{\itemsep}{0pt} 
		\item[\bf 1.] For any set $\mathcal{A} \subseteq [\mathbb{N}]^{\omega}$, the following statements are equivalent:
		\begin{enumerate}
		\setlength{\itemsep}{0pt}
			\item[(a)] $\mathcal{A}$ is Ramsey null.
			\item[(b)] $\mathcal{A}$ is nowhere dense.
			\item[(c)] $\mathcal{A}$ is meager.
		\end{enumerate}
		\item[\bf 2.] For any set $\mathcal{A} \subseteq [\mathbb{N}]^{\omega}$, the following statements are equivalent:
		\begin{enumerate}
		\setlength{\itemsep}{0pt}
			\item[(a)] $\mathcal{A}$ is Ramsey.
			\item[(b)] $\mathcal{A}$ has the abstract Baire property.
			\item[(c)] $\mathcal{A}$ has the Baire property.
		\end{enumerate}
		\item[\bf 3.] If $[\mathbb{N}]^{\omega} = \bigcup_{i=1}^{k} \mathcal{P}_{i}$ is any finite partition of $[\mathbb{N}]^{\omega}$, where each piece $\mathcal{P}_{i}$ has the Baire property in the Ellentuck topology, then there is $H\in [\mathbb{N}]^{\omega}$ such that $[H]^{\omega} \subseteq \mathcal{P}_{i}$ for some $i\in\{1,\ldots,k\}$.
	\end{enumerate}
\end{theorem}

\smallskip

\begin{definition} \label{KastanasGame}
	(Kastanas, \cite{Kastanas}). Let $\mathcal{A} \subseteq [\mathbb{N}]^{\omega}$, $N\in[\mathbb{N}]^{\omega}$, and $s\in [N]^{<\omega}$ be given. The \textit{Kastanas game} $\mathcal{J}(\mathcal{A},N,s)$ is defined as follows: 
	\begin{equation*}
		\begin{matrix}
			\textrm{I} &  & M_{0} &  & M_{1} &  & \cdots & & M_{k} & & \cdots & \\ 
			\textrm{II} &  &  & ( n_{0},N_{0} ) &  & ( n_{1},N_{1} ) &  & \cdots & & ( n_{k},N_{k} ) & & \cdots
		\end{matrix}
	\end{equation*}
    Two players \textrm{I} and \textrm{II} take turns playing sequence of sets $\{ M_{k} \}_{k\in\mathbb{N}} \subseteq [\mathbb{N}]^{\omega}$ and $\{ N_{k} \}_{k\in\mathbb{N}} \subseteq [\mathbb{N}]^{\omega}$ respectively, and additionally, the player \textrm{II} is required to play a sequence of integers $\{n_{k}\}_{k\in\mathbb{N}} \subseteq \mathbb{N}$, such that the moves of both players respect the following conditions:
	
    \medskip
    
	\textit{Rules of the game.} For every $k\in\mathbb{N}$:
	\begin{equation*}
		\textbf{(i) } M_{0} \subseteq N/s \;\;\;\;\;\;\;\; \textbf{(ii) } n_{k}\in M_{k} \;\;\;\;\;\;\;\; \textbf{(iii) } N_{k}\subseteq M_{k}/ \{n_{k}\} \;\;\;\;\;\;\;\; \textbf{(iv) } M_{k+1} \subseteq N_{k}
	\end{equation*}
	Let $N_{\infty} \in [s,N]$ be the set given by $N_{\infty} = s\cup \{ n_{k} \,|\, k\in \mathbb{N} \}$. Then, for the game $\mathcal{J}(\mathcal{A},N,s)$ the following is decided:
	\setlist{nolistsep}
	\begin{itemize}
		\setlength{\itemsep}{0pt}	
		\item[$\bullet$] Player \textrm{I} wins the game if $N_{\infty} \in \mathcal{A}$.
		\item[$\bullet$] Player \textrm{II} wins the game if $N_{\infty} \notin \mathcal{A}$.
	\end{itemize}
\end{definition}


\begin{theorem} \label{KastanasTheorem}
	(Kastanas, \cite{Kastanas}). For any set $\mathcal{A} \subseteq [\mathbb{N}]^{\omega}$, the following statements are equivalent:
	\setlist{nolistsep}
	\begin{enumerate}
		\setlength{\itemsep}{0pt}	
		\item[(a)] $\mathcal{A}$ is Ramsey.
		\item[(b)] For all $N\in[\mathbb{N}]^{\omega}$ and $s\in [N]^{<\omega}$, the Kastanas game $\mathcal{J}(\mathcal{A},N,s)$ is determined, this means that one of the players has a winning strategy in the game.
	\end{enumerate}
\end{theorem}

\smallskip

We notice that the Kastanas game $\mathcal{J}(\mathcal{A},N,s)$ on $[\mathbb{N}]^{\omega}$ is equivalent to a similar topological game in which the moves executed by the player \textrm{II} are non-empty basic open sets of the Ellentuck topology, satisfying the rules that we state in definition \ref{abstract-Kastanas-game} for the abstract context. Thus, theorem \ref{KastanasTheorem} is just a particular case of our theorem \ref{Abstract-Kastanas-Theorem}. 

\medskip

According to Todor\v{c}evi\'{c} in \cite{Todorcevic(BookRamsey)}, the \textit{Ellentuck space} is the combinatorial structure $([\mathbb{N}]^{\omega}, \subseteq , r)$, where $[\mathbb{N}]^{\omega}$ is ordered by the inclusion relation $\subseteq$, and $r$ is the function $r : [\mathbb{N}]^{\omega} \times \mathbb{N} \longrightarrow [\mathbb{N}]^{<\omega}$ given by $r (B, n) = r_{n} (B) = s$ if $s\in [\mathbb{N}]^{n}$ and $s\sqsubset B$. Notice that the basis for Ellentuck topology of $[\mathbb{N}]^{\omega}$ is the collection of all sets $[s,B] = \{ A \in [\mathbb{N}]^{\omega} \, | \, A \subseteq B \wedge (\exists\, n\in\mathbb{N}) (r_{n}(A) = s)\}$, with $s\in [\mathbb{N}]^{<\omega}$ and $B\in [\mathbb{N}]^{\omega}$. This alternative presentation of the Ellentuck space allows us to study the Ramsey property in an abstract sense through the combinatorial framework of topological Ramsey space, as will be discussed in the next section.

\section{Topological Ramsey spaces}\label{TRS}

In this section, we formally present the combinatorial structure of a wide class of topological Ramsey spaces as well as the Ramsey property for this context, and we conclude with an exposition of the \textit{abstract Ellentuck theorem}. The background on topological Ramsey spaces throughout this section is taken from \cite{Todorcevic(BookRamsey)}.

\medskip

Consider a \textit{combinatorial structure} of the form
\begin{center}
	$(\mathcal{R}, \leq, r)$
\end{center}
where $\mathcal{R}$ is a non-empty set of objects, $\leq$ is a quasi-order relation on $\mathcal{R}$, and $r$ is a mapping $r: \mathcal{R}\times\mathbb{N} \rightarrow \mathcal{AR}$ giving us the sequence of approximations mappings $r_{n}: \mathcal{R} \rightarrow \mathcal{AR}_{n}$ such that $r_{n} (A) = r(A,n)$ for every $A\in \mathcal{R}$ and $n\in \mathbb{N}$, where the range of $r$ is the set $\mathcal{AR}$ of all the finite approximations of elements of $\mathcal R$, and the range of $r_{n}$ is the set $\mathcal{AR}_{n}$ of all the $n$-approximations, hence $\mathcal{AR} = \bigcup_{n\in \mathbb{N}} \mathcal{AR}_{n}$. The idea is to identify each object $A\in \mathcal{R}$ with the sequence of its respective approximations $\{r_{n}(A)\}_{n\in\mathbb{N}} \subseteq \mathcal{AR}$. 

\medskip

For every object $B\in \mathcal{R}$, the restriction of the structure $\mathcal{R}$ to $B$ is the set $\mathcal{R} \!\restriction\! B$ defined by all objects $A\in \mathcal{R}$ such that $A$ is below $B$ with respect to the quasi-order $\leq$, hence $\mathcal{R} \!\restriction\! B = \{ A\in\mathcal{R} \,|\, A\leq B \}$.

\medskip

Following Todor\v{c}evi\'{c} in \cite{Todorcevic(BookRamsey)}, we will consider combinatorial structures $(\mathcal{R}, \leq , r)$ that satisfy a special list of combinatorial properties, referred to as the axioms of \textit{metrization}, \textit{finitization}, \textit{amalgamation}, and \textit{homogeneity} (also known as the \textit{abstract pigeonhole principle}).

\bigskip

\textbf{A.1. Axioms of Metrization.}
\setlist{nolistsep}
\begin{enumerate}
	\setlength{\itemsep}{0pt}	
	\item[\bf 1.] $r_{0}(A) = \emptyset$ for all $A \in \mathcal{R}$.
	\item[\bf 2.] For $A,B\in \mathcal{R}$, if $A \neq B$ then $r_{n}(A) \neq r_{n}(B)$ for some $n\in \mathbb{N}$.
	\item[\bf 3.] For $A,B\in \mathcal{R}$, if $r_{n}(A) = r_{m}(B)$ then $n=m$ and $r_{k}(A) = r_{k}(B)$ for each $k \leq n$.
\end{enumerate}

\bigskip

Let $A\in \mathcal{R}$ and $a,b\in \mathcal{AR}$ be given. Then, we write $a\sqsubset A$, and say that $a$ \textit{is initial segment of} $A$, if there is an $n\in \mathbb{N}$ such that $a=r_{n}(A)$. Likewise, we write $a\sqsubseteq b$, and say that $a$ \textit{is initial segment of} $b$, if there are $B\in \mathcal{R}$ and $n,m\in \mathbb{N}$ with $n\leq m$ such that $a=r_{n}(B)$ and $b=r_{m}(B)$. Similarly, we write $a\sqsubset b$, and say that $a$ \textit{is proper initial segment of} $b$, if $a\sqsubseteq b$ and $a\neq b$. Lastly, the \textit{length of $a$}, denoted by $|a|$, is the unique $n\in\mathbb{N}$ such that $a=r_{n}(B)$ for some $B\in \mathcal{R}$.

\bigskip

\textbf{A.2. Axioms of Finitization.} There is a quasi-order $\leq_{\text{fin}}$ on $\mathcal{AR}$ such that:
\setlist{nolistsep}
\begin{enumerate}
	\setlength{\itemsep}{0pt}	
	\item[\bf 1.] For all $b\in \mathcal{AR}$, the set of approximations $\{ a\in \mathcal{AR} \, | \, a \leq_{\text{fin}} b \}$ is finite.
	\item[\bf 2.] Let $A,B\in \mathcal{R}$, then $A \leq B$ if and only if for all $n\in \mathbb{N}$ there is $m\in \mathbb{N}$ such that $r_{n}(A) \leq_{\text{fin}} r_{m}(B)$.
	\item[\bf 3.] For all $a,b,c\in \mathcal{AR}$, if $a\sqsubset b$ and $b\leq_{\text{fin}} c$ then there is $d\in \mathcal{AR}$ such that $d\sqsubset c$ and $a\leq_{\text{fin}} d$. 
\end{enumerate}

\bigskip

For every approximation $a\in \mathcal{AR}$, the set $[a]$ is defined as the family of all objects $A\in \mathcal{R}$ such that $a\sqsubset A$, that is,
\begin{center}
	$[a] = \{ A\in \mathcal{R} \,|\, r_{|a|}(A) = a \}$.
\end{center}
The collection $\{ [a] \,|\, a\in\mathcal{AR} \}$ forms a basis for a topology over $\mathcal{R}$, called the \textit{metrizable topology} of $\mathcal{R}$, whose complete metric $\rho: \mathcal{R} \times \mathcal{R} \rightarrow \mathbb{R}$ is described by $\rho(A,B) = 1/2^{k}$, whenever $A \neq B$ and $k=\min \{ n\in\mathbb{N} \,|\, r_{n}(A) \neq r_{n}(B) \}$.

\medskip

For every approximation $a\in \mathcal{AR}$ and every object $B\in \mathcal{R}$, the set $[a,B]$ is defined as the family of all objects $A\in \mathcal{R}$ such that $a\sqsubset A$ and $A\leq B$, so that
\begin{center}
	$[a,B] = \{ A\in [a] \,|\, A\leq B \} = \{ A\in \mathcal{R} \,|\, A\leq B \wedge r_{|a|}(A) = a \}$.
\end{center}
The collection $\{ [a,B] \,|\, a\in\mathcal{AR} \wedge B\in\mathcal{R} \}$ also forms a basis for a topology over $\mathcal{R}$, called the \textit{Ellentuck topology} of $\mathcal{R}$, which extends and refines the metrizable topology. For simplicity, for each $m\in\mathbb{N}$ we denote: $[m,B] = \{ A\in \mathcal{R} \,|\, A\leq B \wedge r_{m}(A) = r_{m}(B) \} = [r_{m}(B),B]$.

\medskip

The \textit{depth} of an approximation in an object of $\mathcal{R}$ is defined as follows: given $a\in \mathcal{AR}$ and $B\in \mathcal{R}$, the \textit{depth of $a$ in $B$}, denoted by $\mathtt{depth}_{B}(a)$, is 
\begin{equation*}
	\mathtt{depth}_{B}(a) = \begin{cases} \min\{n\in\mathbb{N} \,|\, a\leq_{\text{fin}} r_{n}(B) \} & \text{if such minimum exists.} \\	\infty & \text{otherwise.} \end{cases}
\end{equation*}	

\medskip

For each basic open set $[a,B] \neq \emptyset$ in the Ellentuck topology of $\mathcal{R}$, the restriction of the collection of approximations $\mathcal{AR}$ to $[a,B]$ is the set $\mathcal{AR} \!\restriction\! [a,B] = \{ b\in\mathcal{AR} \,|\, a\sqsubseteq b \wedge \mathtt{depth}_{B}(b) <\infty \}$.

\bigskip

\textbf{A.3. Axioms of Amalgamation.}
\setlist{nolistsep}
\begin{enumerate}
	\setlength{\itemsep}{0pt}
	\item[\bf 1.] Let $a\in \mathcal{AR}$ and $B\in \mathcal{R}$, if $\mathtt{depth}_{B}(a) < \infty$ then $[a,A] \neq \emptyset$ for all $A \in [\mathtt{depth}_{B}(a), B]$, in particular $[a,B] \neq \emptyset$.
		\item[\bf 2.] Let $a\in \mathcal{AR}$ and $A,B\in \mathcal{R}$, if $A \leq B$ and $\mathtt{depth}_{A}(a) \leq \mathtt{depth}_{B}(a) < \infty$ then there is some $A^{\prime} \in [\mathtt{depth}_{B}(a), B]$ such that $[a, A^{\prime}] \subseteq [a,A]$.
\end{enumerate}

\bigskip

Let $[a,A]$ be a basic open set in the Ellentuck topology of $\mathcal{R}$; then, for each $n\in \mathbb{N}$ with $n \geq |a|$, we denote by $r_{n} \text{''}\, [a,A]$ the direct image of the set $[a,A]$ through the $n$-th approximation function $r_{n}$, hence $r_{n} \text{''}\, [a,A] = \{ d\in \mathcal{AR}_{n} \,|\, (\exists D\in [a,A])(d=r_{n}(D)) \} = \{ d\in\mathcal{AR} \!\restriction\! [a,A] \,|\, |d|=n \}$.

\bigskip

\textbf{A.4. Abstract Pigeonhole Principle.} For all $a\in \mathcal{AR}$ and $B\in \mathcal{R}$ such that $\mathtt{depth}_{B}(a) < \infty$, if $\mathcal{O} \subseteq \mathcal{AR}_{|a|+1}$ then there exists some $A\in [\mathtt{depth}_{B}(a),B]$ such that either $r_{|a|+1} \text{''}\, [a,A] \subseteq \mathcal{O}$ or $\mathcal{O} \cap r_{|a|+1} \text{''}\, [a,A] = \emptyset$.

\bigskip

A combinatorial structure $(\mathcal{R},\leq,r)$ is said to be \textit{metrically closed} if $\mathcal{R}$ is a closed subspace of $\mathcal{AR}^{\mathbb{N}}$ with the product topology, where each object of $\mathcal{R}$ has been identified with the sequence of its	initial segments. Naturally, we consider the product topology on $\mathcal{AR}^{\mathbb{N}}$ assuming that $\mathcal{AR}$ has the discrete topology, in which case the subspace topology on $\mathcal{R}$ inherited from $\mathcal{AR}^{\mathbb{N}}$ coincides exactly with the metrizable topology of $\mathcal{R}$.

\medskip

The Ellentuck space $([\mathbb{N}]^{\omega}, \subseteq , r)$ is the quintessential example of a metrically closed combinatorial structure satisfying axioms \textbf{A.1}, \textbf{A.2}, \textbf{A.3}, and \textbf{A.4}. This fact motivates a natural extension of definition \ref{RamseyProperty} to a more abstract context. Accordingly, the Ramsey property for general combinatorial structures $(\mathcal{R},\leq,r)$ is defined as follows:

\smallskip

\begin{definition} \label{RamseySet}
Let $(\mathcal{R},\leq,r)$ be a combinatorial structure. A set $\mathcal{X} \subseteq \mathcal{R}$ is said to be \textit{Ramsey} if for every $[a,A] \neq \emptyset$ there exists some $B\in [a,A]$ such that either $[a,B]\subseteq \mathcal{X}$ or $[a,B] \cap \mathcal{X} = \emptyset$. Moreover, if the second option always holds, then we say that $\mathcal{X}$ is \textit{Ramsey null}.
\end{definition}

\smallskip

It is worth mentioning that, in \cite{Todorcevic(BookRamsey)}, Todor\v{c}evi\'{c} defines a \textit{topological Ramsey space} as a combinatorial structure $(\mathcal{R},\leq,r)$, equipped with the Ellentuck topology, for which every Baire set is Ramsey and every meager set is Ramsey null (see \cite[Definition 5.3]{Todorcevic(BookRamsey)}). Subsequently, Todor\v{c}evi\'{c} proves the \textit{abstract Ellentuck theorem}, which states that if $(\mathcal{R},\leq,r)$ is a metrically closed combinatorial structure endowed with the Ellentuck topology and satisfying axioms \textbf{A.1}, \textbf{A.2}, \textbf{A.3}, and \textbf{A.4}, then $(\mathcal{R},\leq,r)$ forms a topological Ramsey space (see \cite[Theorem 5.4]{Todorcevic(BookRamsey)}).

\smallskip

\begin{definition} 
	(Todor\v{c}evi\'{c}, \cite{Todorcevic(BookRamsey)}). A combinatorial structure $(\mathcal{R},\leq,r)$ endowed with the Ellentuck topology is said to be a \textit{topological Ramsey space} if every subset of $\mathcal{R}$ with the Baire property is Ramsey and every meager subset of $\mathcal{R}$ is Ramsey null.
\end{definition}

\smallskip

In this article, we focus exclusively on combinatorial structures $(\mathcal{R}, \leq , r)$, equipped with the Ellentuck topology, that are metrically closed and satisfy the axioms of metrization, finitization, amalgamation, and homogeneity. This class of structures encompasses a wide range of topological Ramsey spaces with rich combinatorial properties, some of which we examine in detail. There exist examples of topological Ramsey spaces in the sense of Todor\v{c}evi\'{c}'s definition that do not meet the requirements of our framework (see, for instance, \cite[Corollary 7.24 and Remark 7.27]{Todorcevic(BookRamsey)}). We will use the term \textit{axiomatized topological Ramsey space} to refer to those closed topological structures that satisfy axioms \textbf{A.1}, \textbf{A.2}, \textbf{A.3}, and \textbf{A.4}, as stated formally in the following definition.

\smallskip

\begin{definition}
    A metrically closed combinatorial structure $(\mathcal{R},\leq,r)$ is called an \textit{axiomatized topological Ramsey space} if $\mathcal{R}$ is endowed with the Ellentuck topology and $(\mathcal{R},\leq,r)$ satisfies axioms \textbf{A.1}, \textbf{A.2}, \textbf{A.3}, and \textbf{A.4}.
\end{definition}

\smallskip

A topological Ramsey space $(\mathcal{R},\leq,r)$ does not have isolated points if each of its non-empty basic open sets is infinite. In \cite{Todorcevic(BookRamsey)}, Todor\v{c}evi\'{c} proved that for every axiomatized topological Ramsey space $(\mathcal{R},\leq,r)$ without isolated points there exists a set $\mathcal{X} \subseteq \mathcal{R}$ that is not Ramsey, and the use of the axiom of choice is essential in the proof. Furthermore, observe that $(\mathcal{R},\leq,r)$ does not have isolated points whenever the set of approximations $r_{|a|+1} \text{''}\, [a,B]$ is infinite for every basic open set $[a,B] \neq \emptyset$.  

\medskip

An important tool in the study of topological Ramsey spaces is the concept of \textit{fusion sequence}, which is a special sequence formed by basic open sets. Fusion sequences allow for the recursive construction of new objects in the space with specific and often highly structured properties.

\smallskip

\begin{definition} \label{FusionSequence}	
Let $(\mathcal{R},\leq,r)$ be an axiomatized topological Ramsey space. A sequence $\{[n_{k}, Y_{k}]\}_{k\in\mathbb{N}}$ of basic open sets of $\mathcal{R}$ is a \textit{fusion sequence} if:
\begin{itemize}
	\justifying
	\item[$\bullet$] $\{n_{k}\}_{k\in\mathbb{N}} \subseteq \mathbb{N}$ is a strictly increasing sequence. 
	\item[$\bullet$] $Y_{k+1} \in [n_{k},Y_{k}]$ for all $k\in\mathbb{N}$.
\end{itemize}

\medskip

The limit of the fusion sequence $\{[n_{k}, Y_{k}]\}_{k\in\mathbb{N}}$ always exists and it is the unique object $Y_{\infty} \in \mathcal{R}$, denoted by $Y_{\infty}= \lim_{n_{k}} Y_{k}$, such that $Y_{\infty} \in [n_{k},Y_{k}]$ for all $k\in \mathbb{N}$; hence, $r_{n_{k}}(Y_{\infty}) = r_{n_{k}}(Y_{k})$ for all $k\in \mathbb{N}$. In other words, $Y_{\infty}= \lim_{n_{k}} Y_{k}$ if and only if $\{Y_{\infty}\} = \bigcap_{k\in\mathbb{N}} [n_{k}, Y_{k}] = \bigcap_{k\in\mathbb{N}} [r_{n_{k}}(Y_{k})]$.
\end{definition}

\smallskip

We now present the main result of topological Ramsey theory, which corresponds to the natural generalization of theorem \ref{EllentuckTheorem} for the abstract context. In this regard, the abstract Ellentuck theorem establishes the equivalence between the Ramsey property and the Baire property on the combinatorial structure of axiomatized topological Ramsey spaces. In other words, the abstract Ellentuck theorem essentially states that every axiomatized topological Ramsey space is itself a topological Ramsey space.

\smallskip

\begin{theorem} \label{Abstract-Ellentuck-Theorem}
	[\textit{Abstract Ellentuck Theorem}]. (Todor\v{c}evi\'{c}, \cite{Todorcevic(BookRamsey)}). Let $(\mathcal{R},\leq,r)$ be an axiomatized topological Ramsey space. Then:
	\setlist{nolistsep}
	\begin{enumerate}
		\setlength{\itemsep}{0pt} 
		\item[\bf 1.]  For any $\mathcal{X} \subseteq \mathcal{R}$, the following statements are equivalent:
		\begin{enumerate}
			\setlength{\itemsep}{0pt}
			\item[(a)] $\mathcal{X}$ is Ramsey null.
			\item[(b)] $\mathcal{X}$ is nowhere dense.
			\item[(c)] $\mathcal{X}$ is meager.
		\end{enumerate}
		\item[\bf 2.] For any $\mathcal{X} \subseteq \mathcal{R}$, the following statements are equivalent:
		\begin{enumerate}
			\setlength{\itemsep}{0pt}
			\item[(a)] $\mathcal{X}$ is Ramsey.
			\item[(b)] $\mathcal{X}$ has the abstract Baire property.
			\item[(c)] $\mathcal{X}$ has the Baire property.
		\end{enumerate}
		\item[\bf 3.] If $\mathcal{R} =\bigcup_{i=1}^{k} \mathcal{P}_{i}$ is any finite partition of the space $\mathcal{R}$, where each piece $\mathcal{P}_{i}$ has the Baire property, then there is $H\in \mathcal{R}$ such that $\mathcal{R} \!\restriction\! H \subseteq \mathcal{P}_{i}$ for some $i\in\{1,\ldots,k\}$.
	\end{enumerate}
\end{theorem}

\smallskip

The proof of the abstract Ellentuck theorem given by Todor\v{c}evi\'{c} in \cite{Todorcevic(BookRamsey)} makes use of a combinatorial technique known as \textit{abstract combinatorial forcing}, whose essential idea is to force the recursive construction of an object in an axiomatized topological Ramsey space such that it turns out to be homogeneous with respect to a fixed partition of the space that is finite and Baire measurable. In the next section, we present a new proof of the abstract Ellentuck theorem for a large family of axiomatized topological Ramsey spaces in which we do not use combinatorial forcing.


\section{On an alternative proof of the abstract Ellentuck theorem} \label{alternativeproof}

In this section, we introduce a generalization of the axioms of amalgamation to define the class of strong axiomatized topological Ramsey spaces, with the aim of proposing an alternative proof of the abstract Ellentuck theorem for this new broad class of topological Ramsey spaces. 

\medskip

Let $(\mathcal{R},\leq,r)$ be a combinatorial structure equipped with the Ellentuck topology, and let $A,B \in \mathcal{R}$, $b\in\mathcal{AR}$, and $m\in\mathbb{N}$ be such that $B\leq A$ and $\mathtt{depth}_{A}(b) \leq m$. We define the set $\langle b,B \rangle_{m}^{A}$ as follows:
\begin{equation*}
		\langle b,B \rangle_{m}^{A} = \{ X\in [b,B] \;|\; \mathtt{depth}_{A}(r_{|b|+1}(X)) > m \}.
\end{equation*}
Notice that $\langle b,B \rangle_{m}^{A}$ is an open set of $\mathcal{R}$, since $\langle b,B \rangle_{m}^{A} = \bigcup \{ [c,B] \,|\, c\in r_{|b|+1} \text{''}\, [b,B] \wedge \mathtt{depth}_{A}(c)>m \}$. Furthermore, notice that if $\mathtt{depth}_{A}(b) = m$ then $\langle b,B \rangle_{m}^{A} = [b,B]$. 

\bigskip

\textbf{A.3*. Axioms of Strong Amalgamation.}
\setlist{nolistsep}
\begin{enumerate}
	\setlength{\itemsep}{0pt}
	\item[\bf 1.] Let $b\in \mathcal{AR}$ and $A\in \mathcal{R}$, if $\mathtt{depth}_{A}(b)\leq m < \infty$ then $\langle b,B \rangle_{m}^{A} \neq \emptyset$ for all $B \in [m, A]$, in particular $\langle b,A \rangle_{m}^{A} \neq \emptyset$.
	\item[\bf 2.] Let $b\in \mathcal{AR}$ and $A,B\in \mathcal{R}$, if $B \leq A$ and $\mathtt{depth}_{B}(b) \leq \mathtt{depth}_{A}(b) \leq m <\infty$ then there is some $B^{\prime} \in [m, A]$ such that $\langle b,B^{\prime} \rangle_{m}^{A} \subseteq \langle b,B \rangle_{m}^{A}$.
\end{enumerate}

\bigskip

We point out that the axioms of strong amalgamation \textbf{A.3*} imply the axioms of amalgamation \textbf{A.3}, since taking  $m=\mathtt{depth}_{A}(b)$, we can see that \textbf{A.3} is a particular case of \textbf{A.3*}.

\smallskip
		
\begin{definition} \label{Strong-topological-Ramsey-space}
	A metrically closed combinatorial structure $(\mathcal{R},\leq,r)$ is called a \textit{strong axiomatized topological Ramsey space} if $\mathcal{R}$ is endowed with the Ellentuck topology and $(\mathcal{R},\leq,r)$ satisfies axioms \textbf{A.1}, \textbf{A.2}, \textbf{A.3*}, and \textbf{A.4}. 
\end{definition}

\smallskip

It should be noted that an axiomatized topological Ramsey space $(\mathcal{R},\leq,r)$ satisfies axiom \textbf{A.3*.1} if and only if, for every basic open set $[a,B] \neq \emptyset$, the set of approximations $r_{|a|+1} \text{''}\, [a,B]$ is infinite. Therefore, if the space $(\mathcal{R},\leq,r)$ is strong, it must not have isolated points.

\medskip

Additionally, it is clear that every strong axiomatized topological Ramsey space is itself an axiomatized topological Ramsey space, since axioms \textbf{A.3*} imply axioms \textbf{A.3}. It is also known that there exist examples of topological Ramsey spaces without isolated points that are not axiomatized (see, for instance, \cite{Dobrinen-Zucker}). However, we do not yet know the answer to the following question: is there any axiomatized topological Ramsey space without isolated points that is not strong?

\medskip

Next, we present an alternative proof of the abstract Ellentuck theorem restricted to the collection of strong axiomatized topological Ramsey spaces. The new proof of theorem \ref{Abstract-Ellentuck-Theorem} that we propose here is inspired by a short proof of theorem \ref{EllentuckTheorem} given by Matet in \cite{Matet2}. 

\medskip
 
Let $(\mathcal{R},\leq,r)$ be any axiomatized topological Ramsey space. Given $\mathcal{Q} \subseteq \mathcal{R}$ and $b\in \mathcal{AR}$, we define the set $\mathcal{S}_{b}^{\mathcal{Q}} \subseteq \mathcal{R}$ as the collection of all objects $D\in \mathcal{R}$ such that $\mathtt{depth}_{D}(b) < \infty$ and $[b,D] \subseteq \mathcal{Q}$, that is,
\begin{center}
	$\mathcal{S}_{b}^{\mathcal{Q}} = \{ D\in \mathcal{R} \,|\, \emptyset \neq [b,D] \subseteq \mathcal{Q} \}$.
\end{center}

\smallskip
		
\begin{lemma} \label{Small-Lemma}
	Let $(\mathcal{R},\leq,r)$ be a strong axiomatized topological Ramsey space. For each $\mathcal{Q} \subseteq \mathcal{R}$, $m\in\mathbb{N}$, $A\in\mathcal{R}$, and $b,d\in\mathcal{AR}$, if $\mathtt{depth}_{A}(b) \leq \mathtt{depth}_{A}(d) \leq m$ and $b \leq_{\text{fin}} d$, then there is some $B\in[m,A]$ such that either $\langle d,B \rangle_{m}^{A} \subseteq \mathcal{S}_{b}^{\mathcal{Q}}$ or $\langle d,B \rangle_{m}^{A} \cap \mathcal{S}_{b}^{\mathcal{Q}} = \emptyset$.
\end{lemma}

\begin{proof}
	If $\langle d,A \rangle_{m}^{A} \cap \mathcal{S}_{b}^{\mathcal{Q}} = \emptyset$ then we can simply take $B=A$. Otherwise, if $\langle d,A \rangle_{m}^{A} \cap \mathcal{S}_{b}^{\mathcal{Q}} \neq \emptyset$ then we fix $C\in \langle d,A \rangle_{m}^{A} \cap \mathcal{S}_{b}^{\mathcal{Q}}$, hence $C\in [d,A]$ and $\emptyset \neq [b,C] \subseteq \mathcal{Q}$. Now, by \textbf{A.3*} there exists some $B\in [m,A]$ such that $\emptyset \neq \langle d,B \rangle_{m}^{A} \subseteq \langle d,C \rangle_{m}^{A}$. Finally, notice that $\langle d,B \rangle_{m}^{A} \subseteq \mathcal{S}_{b}^{\mathcal{Q}}$, since every  $D\in \langle d,B \rangle_{m}^{A} \subseteq \langle d,C \rangle_{m}^{A} \subseteq [d,C]$  satisfies that $\emptyset \neq [b,D] \subseteq [b,C] \subseteq \mathcal{Q}$, therefore $D\in \mathcal{S}_{b}^{\mathcal{Q}}$.
\end{proof}

\smallskip

\begin{lemma} \label{Abstract-Matet-Lemma}
	Let $(\mathcal{R},\leq,r)$ be a strong axiomatized topological Ramsey space, and let $[a,A]\neq \emptyset$ be a basic open set of $\mathcal{R}$. If $\{\mathcal{Q}_{i}\}_{i\in \mathbb{N}} \subseteq \wp(\mathcal{R})$ is any sequence of subsets of $\mathcal{R}$ for which $[a,B] \not\subseteq \bigcap_{i\in\mathbb{N}} \mathcal{Q}_{i}$ for all $B\in [a,A]$, then there are $i\in \mathbb{N}$ and $[c,C]\neq \emptyset$, with $C\in[a,A]$, $[c,C]\subseteq [a,A]$, and $|a|\leq |c| \leq |a|+i$, such that $[d,D] \not\subseteq \mathcal{Q}_{i}$ for each $[d,D]\neq \emptyset$ with $[d,D] \subseteq [c,C]$.
\end{lemma}

\begin{proof}
	Let $[a,A] \neq \emptyset$ and $\{\mathcal{Q}_{i}\}_{i\in \mathbb{N}} \subseteq \wp(\mathcal{R})$ be such that $[a,B] \not\subseteq \bigcap_{i\in\mathbb{N}} \mathcal{Q}_{i}$ for all $B\in [a,A]$. We may assume without loss of generality that $a \sqsubset A$.

    \medskip
		
	First, we recursively construct a fusion sequence $\{[n_{k},A_{k}]\}_{k\in\mathbb{N}}$ such that for all $k\in\mathbb{N}$ we have:
	\setlist{nolistsep}
	\begin{enumerate}
	\setlength{\itemsep}{0pt}
		\item[\it (i)] $A_{0}=A$.
		\item[\it (ii)] $n_{k} = \mathtt{depth}_{A}(a)+k = |a|+k$.
		\item[\it (iii)] For all $i\leq k$ and for each $b,d \in \mathcal{AR} \!\restriction\! [a,A_{k}]$ with $\mathtt{depth}_{A_{k}}(b) \leq \mathtt{depth}_{A_{k}}(d) \leq n_{k}$ and $b \leq_{\text{fin}} d$, the object $A_{k+1} \in [n_{k},A_{k}]$ is such that either $\langle d,A_{k+1} \rangle_{n_{k}}^{A_{k}} \subseteq \mathcal{S}_{b}^{\mathcal{Q}_{i}}$ or $\langle d,A_{k+1} \rangle_{n_{k}}^{A_{k}} \cap \mathcal{S}_{b}^{\mathcal{Q}_{i}} = \emptyset$. 
	\end{enumerate}

    \medskip
    
    Suppose inductively that the basic open set $[n_{k},A_{k}]$ has been constructed. Applying \textbf{A.2}, we consider the finite set $\{ (b,d) \in \mathcal{AR}\!\restriction\![a,A_{k}] \times \mathcal{AR}\!\restriction\![a,A_{k}] \,|\, \mathtt{depth}_{A_{k}}(b) \leq \mathtt{depth}_{A_{k}}(d) \leq n_{k} \wedge b \leq_{\text{fin}} d \}$, for which we fix an arbitrary enumeration $\{ (b_{0},d_{0}), \ldots, (b_{p},d_{p}) \}$ with $p\in\mathbb{N}$. Now, we construct a finite sequence $X_{0} \geq \cdots \geq X_{p+1}$ of elements of $[n_{k},A_{k}]$ such that for all $j\leq p$ we have: 
	\setlist{nolistsep}
	\begin{enumerate}
		\setlength{\itemsep}{0pt}
		\item[\it (iv)] $X_{0}=A_{k}$.
		\item[\it (v)] $X_{j+1} \in [n_{k},X_{j}]$.
		\item[\it (vi)] For all $i\leq k$, either $\langle d_{j},X_{j+1} \rangle_{n_{k}}^{A_{k}} \subseteq \mathcal{S}_{b_{j}}^{\mathcal{Q}_{i}}$ or $\langle d_{j},X_{j+1} \rangle_{n_{k}}^{A_{k}} \cap \mathcal{S}_{b_{j}}^{\mathcal{Q}_{i}} = \emptyset$. 
	\end{enumerate}
	
    \medskip
        
    Suppose that the object $X_{j}\in [n_{k},A_{k}]$ has been defined, thus $r_{n_{k}}(X_{j}) = r_{n_{k}}(A_{k})$, which implies that $\mathtt{depth}_{X_{j}}(b_{j}) = \mathtt{depth}_{A_{k}}(b_{j})$ and $\mathtt{depth}_{X_{j}}(d_{j}) = \mathtt{depth}_{A_{k}}(d_{j})$. Then, applying iteratively lemma \ref{Small-Lemma}, we construct a finite sequence $X_{j}^{0} \geq \cdots \geq X_{j}^{k+1}$ of elements of  $[n_{k},X_{j}] \subseteq [n_{k},A_{k}]$ such that for all $i\leq k$ we have:  
		\setlist{nolistsep}
		\begin{enumerate}
			\setlength{\itemsep}{0pt}
			\item[\it (vii)] $X_{j}^{0}=X_{j}$.
			\item[\it (viii)] $X_{j}^{i+1} \in [n_{k},X_{j}^{i}]$.
			\item[\it (ix)] Either $\langle d_{j},X_{j}^{i+1} \rangle_{n_{k}}^{A_{k}} \subseteq \mathcal{S}_{b_{j}}^{\mathcal{Q}_{i}}$ or $\langle d_{j},X_{j}^{i+1} \rangle_{n_{k}}^{A_{k}} \cap \mathcal{S}_{b_{j}}^{\mathcal{Q}_{i}} = \emptyset$.  
		\end{enumerate}
			
    \medskip
    
    Suppose that the object $X_{j}^{i} \in [n_{k},X_{j}]$ has been defined, in which case $\mathtt{depth}_{X_{j}^{i}}(b_{j}) \leq \mathtt{depth}_{X_{j}^{i}}(d_{j}) \leq n_{k}$; so, by lemma \ref{Small-Lemma}, we deduce that there exists $X_{j}^{i+1} \in [n_{k}, X_{j}^{i}] \subseteq [n_{k},X_{j}] \subseteq [n_{k},A_{k}]$ such that either $\langle d_{j},X_{j}^{i+1} \rangle_{n_{k}}^{X_{j}^{i}} \subseteq \mathcal{S}_{b_{j}}^{\mathcal{Q}_{i}}$ or $\langle d_{j},X_{j}^{i+1} \rangle_{n_{k}}^{X_{j}^{i}} \cap \mathcal{S}_{b_{j}}^{\mathcal{Q}_{i}} = \emptyset$. Since $r_{n_{k}}(X_{j}^{i}) = r_{n_{k}}(X_{j}) = r_{n_{k}}(A_{k})$,  it follows that $\langle d_{j},X_{j}^{i+1} \rangle_{n_{k}}^{X_{j}^{i}} = \langle d_{j},X_{j}^{i+1} \rangle_{n_{k}}^{A_{k}}$; thus, either $\langle d_{j},X_{j}^{i+1} \rangle_{n_{k}}^{A_{k}} \subseteq \mathcal{S}_{b_{j}}^{\mathcal{Q}_{i}}$ or $\langle d_{j},X_{j}^{i+1} \rangle_{n_{k}}^{A_{k}} \cap \mathcal{S}_{b_{j}}^{\mathcal{Q}_{i}} = \emptyset$.

    \medskip

		Let $X_{j+1} \in [n_{k},X_{j}]$ be the object $X_{j+1}=X_{j}^{k+1}$, and based on this, we now define the object $A_{k+1} \in [n_{k},A_{k}]$ by $A_{k+1}=X_{p+1}$. Therefore, for each $i\leq k$ and $j\leq p$ we have that either $\langle d_{j},A_{k+1} \rangle_{n_{k}}^{A_{k}} \subseteq \mathcal{S}_{b_{j}}^{\mathcal{Q}_{i}}$ or $\langle d_{j},A_{k+1} \rangle_{n_{k}}^{A_{k}} \cap \mathcal{S}_{b_{j}}^{\mathcal{Q}_{i}} = \emptyset$, completing the induction on $k\in\mathbb{N}$.

    \medskip
		
		Let $B\in [a,A]$ be the limit  $B= \lim_{n_{k}} A_{k}$  of the fusion sequence $\{[n_{k},A_{k}]\}_{k\in\mathbb{N}}$, so that $B\in [n_{k},A_{k}]$ and $r_{n_{k}}(B) = r_{n_{k}}(A_{k})$ for all $k\in \mathbb{N}$. Since $[a,B] \not\subseteq \bigcap_{i\in\mathbb{N}} \mathcal{Q}_{i}$, then there is some $i\in\mathbb{N}$ such that $[a,B] \not\subseteq \mathcal{Q}_{i}$; so, we choose $E\in [a,B] \setminus \mathcal{Q}_{i}$ and we take $c=r_{\ell}(E)$ with $\ell=\max\{ m\in\mathbb{N} \,|\, \mathtt{depth}_{B}(r_{m}(E)) \leq n_{i} \}$, obtaining $|a|\leq |c| \leq |a|+i$. Thus, $a \sqsubseteq c \sqsubset E$ and $E \in [c,B]$, consequently we have that $[c,B] \not\subseteq \mathcal{Q}_{i}$; in fact, it is also true that $[c,E] \not\subseteq \mathcal{Q}_{i}$.

        \medskip
		
		For each $b\in \mathcal{AR} \!\restriction\! [c,E]$, we proceed to consider the family of approximations $T_{b} \subseteq r_{|b|+1} \text{''}\, [b,E]$ as well as the collection of objects $\mathcal{K}_{b} \subseteq [\mathtt{depth}_{E}(b),E]$ defined as follows:
		\begin{align*}
			T_{b} &= \{ q\in r_{|b|+1} \text{''}\, [b,E]  \;|\; [q,E] \subseteq \mathcal{Q}_{i} \} , \text{ and }
			\\
			\mathcal{K}_{b} &= \{ D\in [\mathtt{depth}_{E}(b),E] \;|\; [b,D] \subseteq \mathcal{Q}_{i} \} = \mathcal{S}_{b}^{\mathcal{Q}_{i}} \cap [\mathtt{depth}_{E}(b),E].
		\end{align*}

        \medskip
        
		\textit{Claim 1}. If $b\in \mathcal{AR} \!\restriction\! [c,E]$ and $D_{1},D_{2} \in [\mathtt{depth}_{E}(b),E]$ be such that $D_{1}\in \mathcal{K}_{b}$ and $D_{2} \leq D_{1}$ then $D_{2}\in \mathcal{K}_{b}$. Indeed, notice that $[b,D_{2}] \subseteq [b,D_{1}] \subseteq \mathcal{Q}_{i}$ whenever $D_{1}\in \mathcal{K}_{b}$ and $D_{2} \leq D_{1}$, so that $D_{2}\in \mathcal{K}_{b}$.

        \medskip
		
		\textit{Claim 2}. For all $b\in \mathcal{AR} \!\restriction\! [c,E]$, either $\mathcal{K}_{b}= [\mathtt{depth}_{E}(b),E]$ or $\mathcal{K}_{b}= \emptyset$. Indeed, let $b\in \mathcal{AR} \!\restriction\! [c,E]$ be fixed and put $d= r_{\mathtt{depth}_{E}(b)}(E)$, hence $b \leq_{\text{fin}} d$ and $[\mathtt{depth}_{E}(b),E] = [d,E]$. Now, suppose that $\mathcal{K}_{b} \neq \emptyset$ and let $D\in[d,E]$ be such that $D\in \mathcal{K}_{b}$, which implies that  $D\in \mathcal{S}_{b}^{\mathcal{Q}_{i}}$. We know that $B = \lim_{n_{k}} A_{k}$, so we can take $k\in\mathbb{N}$ such that $n_{k} = \max \{ \mathtt{depth}_{B}(d),n_{i} \}$. Thus $i\leq k$ and $\mathtt{depth}_{B}(b) \leq \mathtt{depth}_{B}(d) \leq n_{k}$, and since $B\in [n_{k+1},A_{k+1}] \subseteq [n_{k},A_{k}]$, we also have that $\mathtt{depth}_{B}(b) = \mathtt{depth}_{A_{k}}(b)$ and $\mathtt{depth}_{B}(d) = \mathtt{depth}_{A_{k}}(d)$. Therefore, we deduce that either $\langle d,A_{k+1} \rangle_{n_{k}}^{A_{k}} \subseteq \mathcal{S}_{b}^{\mathcal{Q}_{i}}$ or $\langle d,A_{k+1} \rangle_{n_{k}}^{A_{k}} \cap \mathcal{S}_{b}^{\mathcal{Q}_{i}} = \emptyset$, and since $r_{n_{k}}(B) = r_{n_{k}}(A_{k+1}) = r_{n_{k}}(A_{k})$ and $B \leq A_{k+1} \leq A_{k}$, then we also obtain that either $\langle d,B \rangle_{n_{k}}^{B} \subseteq \mathcal{S}_{b}^{\mathcal{Q}_{i}}$ or $\langle d,B \rangle_{n_{k}}^{B} \cap \mathcal{S}_{b}^{\mathcal{Q}_{i}} = \emptyset$. However, notice that $D\in \langle d,B \rangle_{n_{k}}^{B} \cap \mathcal{S}_{b}^{\mathcal{Q}_{i}}$, so it must necessarily be true that $\langle d,B \rangle_{n_{k}}^{B} \subseteq \mathcal{S}_{b}^{\mathcal{Q}_{i}}$, hence we infer that $E\in \mathcal{S}_{b}^{\mathcal{Q}_{i}}$ because $E\in \langle d,B \rangle_{n_{k}}^{B}$. We then conclude that $E\in \mathcal{K}_{b}$, and consequently $\mathcal{K}_{b} = [\mathtt{depth}_{E}(b),E]$. 

        \medskip
		
		\textit{Claim 3}. For all $b\in \mathcal{AR} \!\restriction\! [c,E]$, it is true that $\mathcal{K}_{b} \neq \emptyset$ if and only if $T_{b} = r_{|b|+1} \text{''}\, [b,E]$. Indeed, if $\mathcal{K}_{b} \neq \emptyset$ then $\mathcal{K}_{b}= [\mathtt{depth}_{E}(b),E]$, in which case $E\in \mathcal{K}_{b}$ and hence $[b,E] \subseteq \mathcal{Q}_{i}$, thus $[q,E] \subseteq [b,E]\subseteq \mathcal{Q}_{i}$ for each $q\in r_{|b|+1} \text{''}\, [b,E]$, as a result we obtain that $T_{b} = r_{|b|+1} \text{''}\, [b,E]$. Conversely, if $T_{b} = r_{|b|+1} \text{''}\, [b,E]$, then $[b,E] \subseteq \bigcup_{q\in T_{b}} [q,E] \subseteq \mathcal{Q}_{i}$, therefore $E\in \mathcal{K}_{b}$ and in particular $\mathcal{K}_{b} \neq \emptyset$.

        \medskip
		
		\textit{Claim 4}. For every $b\in \mathcal{AR} \!\restriction\! [c,E]$, if $\mathcal{K}_{b} \neq \emptyset$ then $b \neq c$ and $b\in T_{b^{-}}$, where $c\sqsubseteq b^{-} \sqsubset b$ with $|b^{-}|=|b|-1$. Indeed, since $[c,E] \not\subseteq \mathcal{Q}_{i}$ then $E\notin \mathcal{K}_{c}$, in which case $\mathcal{K}_{c} \neq [\mathtt{depth}_{E}(c),E] =[c,E]$ and hence $\mathcal{K}_{c} = \emptyset$. Consequently, if $\mathcal{K}_{b} \neq \emptyset$ then $b\neq c$ and also $\mathcal{K}_{b} = [\mathtt{depth}_{E}(b),E]$, so that $E\in \mathcal{K}_{b}$ and hence $[b,E] \subseteq \mathcal{Q}_{i}$, therefore $b\in T_{b^{-}}$ where $c\sqsubseteq b^{-} \sqsubset b$ with $|b^{-}|=|b|-1$. 

        \medskip
		
		After that stage, we recursively construct a new fusion sequence $\{[m_{k},M_{k}]\}_{k\in\mathbb{N}}$ such that for all $k\in\mathbb{N}$ we have:
	\setlist{nolistsep}
		\begin{enumerate}
			\setlength{\itemsep}{0pt}
			\item[\it (i)] $M_{0}=E$.
			\item[\it (ii)] $m_{k} = \mathtt{depth}_{E}(c)+k = |c|+k$.
			\item[\it (iii)] For all $b\in\mathcal{AR} \!\restriction\! [c,E]$ with $\mathtt{depth}_{M_{k}}(b) = m_{k}$, the object $M_{k+1} \in [m_{k},M_{k}]$ is such that $T_{b} \cap r_{|b|+1} \text{''}\, [b,M_{k+1}] = \emptyset$.
		\end{enumerate}

        \medskip
        
		By induction, suppose that the basic open set $[m_{k},M_{k}]$ has been constructed. Applying \textbf{A.2}, we consider any enumeration of the finite set $\{ b\in\mathcal{AR} \!\restriction\! [c,E] \,|\, \mathtt{depth}_{M_{k}}(b) = m_{k} \} = \{ b_{0},\ldots,b_{p} \}$ with $p\in\mathbb{N}$. Now, we construct a finite sequence $Z_{0} \geq \cdots \geq Z_{p+1}$ of elements of $[m_{k},M_{k}]$ such that for each $j\leq p$ we have:
		\setlist{nolistsep}
		\begin{enumerate}
			\setlength{\itemsep}{0pt}
			\item[\it (iv)] $Z_{0}=M_{k}$.
			\item[\it (v)] $Z_{j+1} \in [m_{k},Z_{j}]$.
			\item[\it (vi)] $T_{b_{j}} \cap r_{|b_{j}|+1} \text{''}\, [b_{j},Z_{j+1}] = \emptyset$.  
		\end{enumerate}

        \medskip
        
        Suppose that the object $Z_{j} \in [m_{k},M_{j}]$ has been defined, and notice that $\mathtt{depth}_{Z_{j}}(b_{j}) = \mathtt{depth}_{M_{k}}(b_{j}) = m_{k}$, then applying the abstract pigeonhole principle \textbf{A.4} to the family $T_{b_{j}}$, we deduce that there exists some $Z_{j+1} \in [m_{k},Z_{j}]$ such that either $r_{|b_{j}|+1} \text{''}\, [b_{j},Z_{j+1}] \subseteq T_{b_{j}}$ or $T_{b_{j}} \cap r_{|b_{j}|+1} \text{''}\, [b_{j},Z_{j+1}] = \emptyset$. However, the first alternative is impossible, since $r_{|b_{j}|+1} \text{''}\, [b_{j},Z_{j+1}] \subseteq T_{b_{j}}$ implies that $[b_{j},Z_{j+1}] \subseteq \mathcal{Q}_{i}$, hence $\mathcal{K}_{b_{j}} \neq \emptyset$ because $Z_{j+1} \in \mathcal{K}_{b_{j}}$, so that $b_{j} \neq c$ and $b_{j} \in T_{b_{j}^{-}}$; in fact, notice that $b_{j} \in T_{b_{j}^{-}} \cap r_{|b_{j}|} \text{''}\, [b_{j}^{-},Z_{j}]$, which is contradictory. Therefore, we must have $T_{b_{j}} \cap r_{|b_{j}|+1} \text{''}\, [b_{j},Z_{j+1}] = \emptyset$.

        \medskip
		
		We now define the object $M_{k+1} \in [m_{k},M_{j}]$ by $M_{k+1}=Z_{p+1}$; thus, for all $j \leq p$ we have that $T_{b_{j}} \cap r_{|b_{j}|+1} \text{''}\, [b_{j},M_{k+1}] = \emptyset$. In this way, we have completed the induction on $k\in\mathbb{N}$.

        \medskip
		
		Finally, let $C\in [c,E] \subseteq [c,B] \subseteq [a,A]$ be the limit $C =\lim_{m_{k}} M_{k}$ of the fusion sequence $\{[m_{k},M_{k}]\}_{k\in\mathbb{N}}$, so that $C\in [m_{k},M_{k}]$ and $r_{m_{k}}(C) = r_{m_{k}}(M_{k})$ for all $k\in \mathbb{N}$. Now, if $[d,D] \neq \emptyset$ is any basic open set of $\mathcal{R}$ such that $[d,D] \subseteq [c,C] \subseteq [c,E]$, then we have that $d\in \mathcal{AR} \!\restriction\! [c,C] \subseteq \mathcal{AR} \!\restriction\! [c,E]$. Furthermore, we can assume that $D\leq E$, in which case there is some $D^{\prime} \in [\mathtt{depth}_{E}(d),E]$ such that $\emptyset \neq [d,D^{\prime}] \subseteq [d,D]$, in view of \textbf{A.3}. In closing, let $k\in \mathbb{N}$ be such that $m_{k} = \mathtt{depth}_{C}(d) = \mathtt{depth}_{M_{k}}(d)$, thus we infer that $T_{d} \cap r_{|d|+1} \text{''}\, [d,M_{k+1}] = \emptyset$, and since $M_{k+1} \leq E$, then we deduce that $T_{d} \neq r_{|d|+1} \text{''}\, [d,E]$, for this reason  $\mathcal{K}_{d} = \emptyset$ and hence $D^{\prime}\notin \mathcal{K}_{d}$, so that $[d,D^{\prime}] \not\subseteq \mathcal{Q}_{i}$ because $D^{\prime}\notin \mathcal{S}_{d}^{\mathcal{Q}_{i}}$. Therefore, we ultimately conclude that $[d,D] \not\subseteq \mathcal{Q}_{i}$ whenever $\emptyset \neq [d,D] \subseteq [c,C]$. 	  		
	\end{proof}

    \smallskip

Applying the previous lemma, we finally conclude with our alternative proof of the abstract Ellentuck theorem for the class of strong axiomatized topological Ramsey spaces. With reference to the statement of theorem \ref{Abstract-Ellentuck-Theorem}, its first item will be proved in proposition \ref{Abstract_Ellentuck_1}, its second item will be shown in proposition \ref{Abstract_Ellentuck_2}, and its third item will be verified in corollary \ref{Abstract_Ellentuck_3}.

\smallskip

\begin{proposition} \label{Abstract_Ellentuck_1}
	Let $(\mathcal{R},\leq,r)$ be a strong axiomatized topological Ramsey space. Then, for any $\mathcal{X} \subseteq \mathcal{R}$, the following statements are equivalent:
	\setlist{nolistsep}
		\begin{enumerate}
			\setlength{\itemsep}{0pt}
			\item[(a)] $\mathcal{X}$ is Ramsey null.
			\item[(b)] $\mathcal{X}$ is nowhere dense.
			\item[(c)] $\mathcal{X}$ is meager.
		\end{enumerate}
\end{proposition}

	\begin{proof}
		Notice that it is immediate that every Ramsey null set is also a nowhere dense set and hence is itself a meager set. So, let us verify that every meager set is  Ramsey null. Indeed, if $\mathcal{X} \subseteq \mathcal{R}$ is a meager set, then there is a sequence $\{ \mathcal{Q}_{i} \}_{i\in \mathbb{N}} \subseteq \wp(\mathcal{R})$ of nowhere dense sets such that $\mathcal{X} = \bigcup_{i\in \mathbb{N}} \mathcal{Q}_{i}$; thus, by considering the sequence $\{\mathcal{Q}_{i}^{\complement} \}_{i\in \mathbb{N}}$, we have that $\bigcap_{i\in\mathbb{N}} \mathcal{Q}_{i}^{\complement} = ( \bigcup_{i\in\mathbb{N}} \mathcal{Q}_{i} )^{\complement} = \mathcal{X}^{\complement}$. Now, by reduction to the absurd, suppose that $\mathcal{X}$ is not a Ramsey null set, then there is $[a,A]\neq\emptyset$ such that for every $B\in[a,A]$ it is true that $[a,B] \cap \mathcal{X} \neq\emptyset$, so $[a,B] \not\subseteq \mathcal{X}^{\complement}$ and hence $[a,B] \not\subseteq \bigcap_{i\in\mathbb{N}} \mathcal{Q}_{i}^{\complement}$. Therefore, by virtue of lemma \ref{Abstract-Matet-Lemma}, we deduce that there are $i\in \mathbb{N}$ and $[c,C]\neq \emptyset$, with $C\in[a,A]$ and $[c,C]\subseteq [a,A]$, such that for each $[d,D]\neq \emptyset$, if $[d,D] \subseteq [c,C]$ then $[d,D] \not\subseteq \mathcal{Q}_{i}^{\complement}$ and hence $[d,D] \cap \mathcal{Q}_{i} \neq \emptyset$; as a result, we conclude that $\mathcal{Q}_{i}$ is not a nowhere dense set, which is contradictory. 	
	\end{proof}

\smallskip

\begin{proposition} \label{Abstract_Ellentuck_2}
	Let $(\mathcal{R},\leq,r)$ be a strong axiomatized topological Ramsey space. Then, for any $\mathcal{X} \subseteq \mathcal{R}$, the following statements are equivalent:
	\setlist{nolistsep}
	\begin{enumerate}
		\setlength{\itemsep}{0pt}
		\item[(a)] $\mathcal{X}$ is Ramsey.
		\item[(b)] $\mathcal{X}$ has the abstract Baire property.
		\item[(c)] $\mathcal{X}$ has the Baire property.
	\end{enumerate}
\end{proposition}

	\begin{proof}
		As in the proof of proposition \ref{Abstract_Ellentuck_1}, it is immediate that every Ramsey set has the abstract Baire property and hence also has the Baire property. 

        \medskip
		
		Conversely, if $\mathcal{X} \subseteq \mathcal{R}$ has the Baire property, then $\mathcal{X} \triangle \mathcal{O}$ is meager for some open set $\mathcal{O}$, hence $\mathcal{X} \triangle \mathcal{O}$ is nowhere dense by proposition \ref{Abstract_Ellentuck_1}; thus, for each $[a,A]\neq \emptyset$ there is $\emptyset \neq [b,B] \subseteq [a,A]$ such that $[b,B] \cap (\mathcal{X} \triangle \mathcal{O}) = \emptyset$. Now, since $\mathcal{O}$ has the Baire property, there is $\emptyset \neq [c,C] \subseteq [b,B] \subseteq [a,A]$ such that either $[c,C] \subseteq \mathcal{O}$ or $[c,C] \cap \mathcal{O} = \emptyset$. Also, since $[c,C]\subseteq [b,B]$, we have $[c,C] \cap (\mathcal{X} \triangle \mathcal{O}) = \emptyset$. Therefore, when $[c,C] \subseteq \mathcal{O}$ we deduce that $[c,C] \subseteq \mathcal{X}$, and when $[c,C] \cap \mathcal{O} =\emptyset$ we infer that $[c,C] \cap \mathcal{X} =\emptyset$; as a result, we conclude that $\mathcal{X}$ has the abstract Baire property. 

        \medskip
	
		Finally, let us check that any set with the abstract Baire property is in turn a Ramsey set. Indeed, suppose that $\mathcal{X} \subseteq \mathcal{R}$ is not a Ramsey set, then there exists $[a,A] \neq \emptyset$ such that for every $B\in [a,A]$ it is true that $[a,B] \not\subseteq \mathcal{X}$ and $[a,B] \not\subseteq \mathcal{X}^{\complement}$. Now, we consider the eventually constant sequence $\{\mathcal{Q}_{i}\}_{i\in\mathbb{N}} \subseteq \wp(\mathcal{R})$ defined by $\mathcal{Q}_{0} = \mathcal{X}$ and $\mathcal{Q}_{i} = \mathcal{R}$ for each $i>0$, so that $\bigcap_{i\in\mathbb{N}} \mathcal{Q}_{i} =\mathcal{Q}_{0} =\mathcal{X}$ and hence $[a,B] \not\subseteq \bigcap_{i\in\mathbb{N}} \mathcal{Q}_{i}$ for all $B\in [a,A]$. Then, by applying lemma \ref{Abstract-Matet-Lemma} to the sequence $\{\mathcal{Q}_{i}\}_{i\in\mathbb{N}}$, we necessarily conclude that there exists $C\in[a,A]$ such that $[d,D] \not\subseteq \mathcal{X}$ whenever $\emptyset \neq [d,D] \subseteq [a,C]$. Next, we consider the constant sequence $\{\mathcal{Q}_{i}^{*}\}_{i\in\mathbb{N}} \subseteq \wp(\mathcal{R})$ defined by $\mathcal{Q}_{i}^{*} = \mathcal{X}^{\complement}$ for each $i\in\mathbb{N}$, so that $\bigcap_{i\in\mathbb{N}} \mathcal{Q}_{i}^{*} =\mathcal{Q}_{i}^{*} =\mathcal{X}^{\complement}$ and hence $[a,B] \not\subseteq \bigcap_{i\in\mathbb{N}} \mathcal{Q}_{i}^{*}$ for all $B\in [a,C] \subseteq [a,A]$. Then, by applying lemma \ref{Abstract-Matet-Lemma} to the sequence $\{\mathcal{Q}_{i}^{*}\}_{i\in\mathbb{N}}$, we deduce that there is some $[e,E]\neq \emptyset$, with $E\in[a,C]$ and $[e,E]\subseteq [a,C]$, such that $[d,D] \not\subseteq \mathcal{X}^{\complement}$ whenever $\emptyset \neq [d,D] \subseteq [e,E]$. Therefore, the basic open set $[e,E]\neq \emptyset$ satisfies both $[d,D] \not\subseteq \mathcal{X}$ and $[d,D] \not\subseteq \mathcal{X}^{\complement}$ for every $[d,D]\neq \emptyset$ such that $[d,D]\subseteq [e,E]$; thus, we conclude that $\mathcal{X}$ does not have the abstract Baire property. 
	\end{proof}

    \smallskip

\begin{corollary} \label{Abstract_Ellentuck_3}
	Let $(\mathcal{R},\leq,r)$ be a strong axiomatized topological Ramsey space. If $\mathcal{R} =\bigcup_{i=1}^{k} \mathcal{P}_{i}$ is any finite partition of the space $\mathcal{R}$, where each piece $\mathcal{P}_{i}$ has the Baire property, then there is $H\in \mathcal{R}$ such that $\mathcal{R} \!\restriction\! H \subseteq \mathcal{P}_{i}$ for some $i\in\{1,\ldots,k\}$.
\end{corollary}

	\begin{proof}
		For $k=2$, let $\mathcal{R} = \mathcal{P}_{1} \cup \mathcal{P}_{2}$ be any partition of $\mathcal{R}$ into two pieces such that $\mathcal{P}_{1}$ and $\mathcal{P}_{2}$ have the Baire property, being $\mathcal{P}_{2} = \mathcal{R} \setminus \mathcal{P}_{1}$. By virtue of proposition \ref{Abstract_Ellentuck_2} we deduce that $\mathcal{P}_{1}$ is a Ramsey set, then for any basic open set $[\emptyset,A]$ there is $H \in [\emptyset, A]$ such that either $[\emptyset, H] \subseteq \mathcal{P}_{1}$ or $[\emptyset, H] \cap \mathcal{P}_{1} = \emptyset$, where $[\emptyset, H] = \mathcal{R} \!\restriction\! H$. Therefore, the object $H \in \mathcal{R}$ satisfies that either $\mathcal{R} \!\restriction\! H \subseteq \mathcal{P}_{1}$ or $\mathcal{R} \!\restriction\! H \subseteq \mathcal{P}_{2}$. Now, for $k>2$ a completely analogous argument is used and the result is similarly proved by applying induction on $k$.
	\end{proof}

\section{On an abstract version of Kastanas theorem}\label{abstractKastanas}

In this section, we develop an abstract version of Kastanas games in order to provide a game-theoretic characterization of the Ramsey property for combinatorial structures $(\mathcal{R},\leq,r)$. Our goal is to establish necessary and sufficient conditions, formulated through topological games, under which a subset of a selective axiomatized topological Ramsey space has the Ramsey property. We conclude with the \textit{abstract Kastanas theorem}, which encapsulates this characterization and constitutes the main result of our article.

\smallskip

\begin{definition} \label{abstract-Kastanas-game}
	Let $(\mathcal{R},\leq,r)$ be an axiomatized topological Ramsey space, and let $\mathcal{X} \subseteq \mathcal{R}$ and $[a_{0},B_{0}] \neq \emptyset$ be fixed but arbitrary. Then, the \textit{abstract Kastanas game} $\mathcal{G}_{[a_{0},B_{0}]}(\mathcal{X})$ is defined as follows:	
	\begin{equation*}
		\begin{matrix}
			\textrm{I} &  & A_{1} &  & A_{2} &  & \cdots & & A_{k} & & \cdots & \\ 
			\textrm{II} &  &  & [ a_{1},B_{1} ] &  & [ a_{2},B_{2} ] &  & \cdots & & [ a_{k},B_{k} ] & & \cdots
		\end{matrix}
	\end{equation*}
	Two players \textrm{I} and \textrm{II} take turns to play alternately both objects in $\mathcal{R}$ and basic open sets of $\mathcal{R}$, such that the moves of both players respect the following conditions:
	
    \medskip
	
    \textit{Rules of the game.} For every $k\in\mathbb{N}$, both the sequence of objects $\{A_{k}\}_{k\in \mathbb{N}} \subseteq \mathcal{R}$ and $\{B_{k}\}_{k\in \mathbb{N}} \subseteq \mathcal{R}$, as well as the sequence of approximations $\{a_{k}\}_{k\in \mathbb{N}} \subseteq \mathcal{AR}$, must satisfy the following rules:
	\begin{equation*}
		\textbf{(i) } A_{k+1} \in [a_{k},B_{k}] \;\;\;\;\;\;\;\; \textbf{(ii) } a_{k} \sqsubset a_{k+1} \;\;\;\;\;\;\;\; \textbf{(iii) } |a_{k+1}|=|a_{k}|+1 \;\;\;\;\;\;\;\; \textbf{(iv) } B_{k}\in [a_{k},A_{k}]
	\end{equation*}
	Let $B_{\infty} \in [a_{0},B_{0}]$ be the object given by $\{B_{\infty}\} = \bigcap_{k\in\mathbb{N}} [a_{k}]$. Then, for the game $\mathcal{G}_{[a_{0},B_{0}]}(\mathcal{X})$ the following is decided:
	\setlist{nolistsep}
	\begin{itemize}
		\setlength{\itemsep}{0pt}	
		\item[$\bullet$] Player \textrm{I} wins the game if $B_{\infty} \in \mathcal{X}$.
		\item[$\bullet$] Player \textrm{II} wins the game if $B_{\infty} \notin \mathcal{X}$.
	\end{itemize}
\end{definition}

\smallskip

It should be noted that in the definition of the abstract Kastanas game $\mathcal{G}_{[a_{0},B_{0}]}(\mathcal{X})$, we can take $[a_{0},B_{0}] \neq \emptyset$ with $a_{0} \sqsubset B_{0}$, in which case we have that the object $B_{\infty} \in [a_{0},B_{0}]$ is such that $\{B_{\infty}\} = \bigcap_{k\in\mathbb{N}} [a_{k}]$ if and only if $B_{\infty} = \lim_{n_{k}} B_{k}$ where $n_{k}= \mathtt{depth}_{B_{k}}(a_{k})$.

\medskip

A topological game, such as the abstract Kastanas game given in the above definition, is said to be \textit{determined} if one of the players has a winning strategy for the game. The concept of strategy is well known (see, for instance, \cite[chapter 33]{Jech}), so we do not provide a formal definition here, although we can say that a strategy for one of the players is just a function that tells the player what to play at each step depending on the previous moves of both players. Furthermore, a strategy in the game is a winning strategy for one of the players if, using this strategy, the outcome of the game is a win for this player independently of the moves of the other player. 

\medskip 

We begin with the following key result regarding axiomatized topological Ramsey spaces, which characterizes when player \textrm{I} has a winning strategy in the abstract Kastanas game.

\smallskip

\begin{proposition} \label{strategy-I}
	Let $(\mathcal{R},\leq,r)$ be an axiomatized topological Ramsey space, and let $\mathcal{X} \subseteq \mathcal{R}$ and $[a,M] \neq \emptyset$ be given. Then, the player \textrm{I} has a winning strategy in the game $\mathcal{G}_{[a,M]}(\mathcal{X})$ if and only if there is $H\in [a,M]$ such that $[a,H] \subseteq \mathcal{X}$.
\end{proposition}

\begin{proof} 
	Given an axiomatized topological Ramsey space $(\mathcal{R},\leq,r)$, we fix a set of objects $\mathcal{X} \subseteq \mathcal{R}$ and a basic open set $[a,M] \neq \emptyset$, for which we assume without loss of generality that $a \sqsubset M$. So, we proceed to consider the abstract Kastanas game $\mathcal{G}_{[a,M]}(\mathcal{X})$.

    \medskip
	
	First of all, suppose that there is an object $H\in [a,M]$ such that $[a,H] \subseteq \mathcal{X}$, in which case it is natural to define a winning strategy for the player \textrm{I} in the game $\mathcal{G}_{[a,M]}(\mathcal{X})$, namely the player \textrm{I} must select  the object $H$ as the first move. Indeed, consider any complete run in the game $\mathcal{G}_{[a,M]}(\mathcal{X})$ of the form 

    \medskip
	
	\begin{minipage}[c]{0.25\linewidth}
		\begin{center}
			$\mathcal{G}_{[a,M]}(\mathcal{X})$
		\end{center}
	\end{minipage}
	\begin{minipage}[c]{0.75\linewidth}
		\begin{equation*}
			\begin{matrix}
				\textrm{I} &  & H &  & A_{2} & & A_{3} & & \cdots & \\ 
				\textrm{II} &  &  & [ a_{1}, B_{1} ] &  & [ a_{2}, B_{2} ] & & [ a_{3}, B_{3} ] &  & \cdots
			\end{matrix}
		\end{equation*}
	\end{minipage}

    \medskip
	
	So, notice that the object $B_{\infty}\in [a,M]$ given by $\{B_{\infty}\} = \bigcap_{k=1}^{\infty} [a_{k}]$ is such that $B_{\infty}\in [a_{k},B_{k}]$ for every $k\in\mathbb{N}-\{0\}$, consequently $B_{\infty} \in [a,H]$ and hence $B_{\infty} \in \mathcal{X}$; therefore, the player \textrm{I} wins the game $\mathcal{G}_{[a,M]}(\mathcal{X})$.

    \medskip
	
	Conversely, suppose that the player \textrm{I} has a winning strategy $\sigma$ in the game $\mathcal{G}_{[a,M]}(\mathcal{X})$, and we consider the strictly increasing sequence $\{n_{k}\}_{n\in\mathbb{N}} \subseteq \mathbb{N}$ described by $n_{k} = \mathtt{depth}_{M}(a)+k = |a|+k$ for each $k\in\mathbb{N}$. Now, our goal is to recursively construct a fusion sequence $\{ [n_{k},A_{k}] \}_{k\in\mathbb{N}}$ such that its limit $H=\lim_{n_{k}} A_{k}$ will satisfy that $[a,H] \subseteq \mathcal{X}$. 

    \medskip
	
	For the first stage of the process, we put $A_{0}\in \mathcal{R}$ and $a_{0}\in \mathcal{AR}$ as $A_{0}=M$ and $a_{0} = r_{n_{0}}(A_{0})=a$. Next, let $A_{1}\in [a_{0},A_{0}]$ be the object determined by $A_{1}= \sigma \langle \emptyset \rangle$, so that $A_{1}$ is the first move of the player \textrm{I} in the game $\mathcal{G}_{[a,M]}(\mathcal{X})$ using the winning strategy $\sigma$, in which case we take the approximation $a_{1} \in \mathcal{AR}$ given by $a_{1} = r_{n_{1}}(A_{1})$, then $a_{0} \sqsubset a_{1}$ and $|a_{1}|=|a_{0}|+1$.

    \medskip
	
	A detailed description of the next stage of the construction will help the reader to go through the general inductive step. In the next stage of the process, by \textbf{A.2.1} we can take the finite set $\{b\in \mathcal{AR} \, | \, a \sqsubset b \wedge \mathtt{depth}_{A_{1}}(b)=n_{1} \} \subseteq \{b\in \mathcal{AR} \, | \, b \leq_{\text{fin}} a_{1} \}$, for which we fix an enumeration $\{ b_{1}^{0}, b_{1}^{1}, \ldots, b_{1}^{p_{1}} \}$ with $p_{1}\in \mathbb{N}$, where $b_{1}^{0}=b_{1}^{p_{1}}=a_{1}$ and $b_{1}^{i} \neq b_{1}^{j}$ for every $1\leq i<j \leq p_{1}$. Next, we recursively construct a finite sequence of objects $C_{1}^{0}, B_{1}^{0}, A_{1}^{0}, \ldots, C_{1}^{p_{1}}, B_{1}^{p_{1}}, A_{1}^{p_{1}}$ in $[a,A_{1}]$, putting initially $C_{1}^{0} =B_{1}^{0} =A_{1}^{0} =A_{1}$. Afterwards, inductively suppose that $C_{1}^{m-1} \geq B_{1}^{m-1} \geq A_{1}^{m-1}$ have been defined for some $1\leq m\leq p_{1}$, where $C_{1}^{m-1} \in [a_{1}]$ and $B_{1}^{m-1}, A_{1}^{m-1} \in [b_{1}^{m-1}]$. Notice that $r_{n_{1}}(C_{1}^{m-1})=a_{1}$, hence $\mathtt{depth}_{A_{1}^{m-1}}(b_{1}^{m-1}) \leq \mathtt{depth}_{C_{1}^{m-1}}(b_{1}^{m-1}) = n_{1}$; thus, by \textbf{A.3.2} there is $C_{1}^{m}\in [a_{1}, C_{1}^{m-1}]$ such that $[b_{1}^{m-1},C_{1}^{m}] \subseteq [b_{1}^{m-1},A_{1}^{m-1}]$. Now, since $r_{n_{1}}(C_{1}^{m})=a_{1}$ then $\mathtt{depth}_{C_{1}^{m}}(b_{1}^{m}) = n_{1}$; thus, by \textbf{A.3.1} we have that $[b_{1}^{m},C_{1}^{m}] \neq \emptyset$, so we take $B_{1}^{m} \in [b_{1}^{m},C_{1}^{m}]$. Later, we consider $A_{1}^{m} \in [b_{1}^{m},B_{1}^{m}]$ determined by $A_{1}^{m} = \sigma \langle A_{1}; [b_{1}^{m},B_{1}^{m}] \rangle$, so that $A_{1}^{m}$ is the second move of the player \textrm{I} in the game $\mathcal{G}_{[a,M]}(\mathcal{X})$ using the winning strategy $\sigma$, in response to the throw $[b_{1}^{m},B_{1}^{m}]$ executed by the player \textrm{II}. As a result, we obtain the desired objects $C_{1}^{m} \geq B_{1}^{m} \geq A_{1}^{m}$ for all $m\leq p_{1}$. Finally, for the last case $b_{1}^{p_{1}}=a_{1}$, we have that $A_{1}^{p_{1}} \in [a_{1}]$ and hence $r_{n_{1}}(A_{1}^{p_{1}}) = a_{1}$; consequently, we define $A_{2}\in [a_{1},A_{1}]$ as the object $A_{2} = A_{1}^{p_{1}}$, and we take $a_{2}\in\mathcal{AR}$ given by $a_{2}=r_{n_{2}}(A_{2})$, then $a_{1} \sqsubset a_{2}$ and $|a_{2}|=|a_{1}|+1$.

    \medskip
	
	In this order of ideas, by induction we fix $k\in\mathbb{N}-\{0\}$ and suppose that the $j$-th stage of the process has been constructed for each $j \in \{0,\ldots,k-1\}$, where we have defined a finite sequence of basic open sets $\{[n_{j},A_{j}]\}_{j=0}^{k}$ such that $[n_{0},A_{0}] \supseteq \cdots \supseteq [n_{k},A_{k}]$, which has associated a finite sequence of approximations $\{a_{j}\}_{j=0}^{k}$ described by $a_{j} = r_{n_{j}}(A_{j})$, hence $a_{0} \sqsubset \cdots \sqsubset a_{k}$ and $|a_{j+1}| = |a_{j}|+1$.

    \medskip

	Now, for the $k$-th stage of the process, we consider the set of approximations $\{ b\in \mathcal{AR} \, | \, a\sqsubset b \wedge \mathtt{depth}_{A_{k}}(b)=n_{k} \}$,  which corresponds to a finite set by virtue of \textbf{A.2.1}, since $\{ b\in \mathcal{AR} \, | \, a\sqsubset b \wedge \mathtt{depth}_{A_{k}}(b)=n_{k} \} \subseteq \{ b\in \mathcal{AR} \, | \, b \leq_{\text{fin}} a_{k} \}$; thus, we fix an enumeration $\{ b\in \mathcal{AR} \, | \, a\sqsubset b \wedge \mathtt{depth}_{A_{k}}(b)=n_{k} \} = \{ b_{k}^{0}, b_{k}^{1}, \ldots, b_{k}^{p_{k}} \}$ with $p_{k}\in \mathbb{N}$, such that $b_{k}^{0}=b_{k}^{p_{k}}=a_{k}$ and $b_{k}^{i} \neq b_{k}^{j}$ for every $1\leq i<j \leq p_{k}$. 

    \medskip
	
	Next, we recursively define a finite sequence of objects $C_{k}^{0}, B_{k}^{0}, A_{k}^{0}, C_{k}^{1}, B_{k}^{1}, A_{k}^{1}, \ldots, C_{k}^{p_{k}}, B_{k}^{p_{k}}, A_{k}^{p_{k}}$ in $[a,A_{k}]$, such that it satisfies the following properties:
	\setlist{nolistsep}
	\begin{enumerate}
		\setlength{\itemsep}{0pt}	
		\item[\it (i)] $C_{k}^{0}= B_{k}^{0}= A_{k}^{0}= A_{k}$.
		\item[\it (ii)] $C_{k}^{m}\geq B_{k}^{m}\geq A_{k}^{m}$ for every $m\in \{0,\ldots,p_{k}\}$.
		\item[\it (iii)] $r_{n_{k}}(C_{k}^{m}) = r_{n_{k}}(A_{k}) = a_{k}$ for every $m\in \{0,\ldots,p_{k}\}$.
		\item[\it (iv)] $C_{k}^{m+1} \in [a_{k}, C_{k}^{m} ]$ for every $m\in \{0,\ldots,p_{k}-1\}$.
		\item[\it (v)] $[b_{k}^{m}, C_{k}^{m+1}] \subseteq [b_{k}^{m}, A_{k}^{m}]$ for every $m\in \{0,\ldots,p_{k}-1\}$.
		\item[\it (vi)] $B_{k}^{m} \in [b_{k}^{m}, C_{k}^{m}]$ for every $m\in \{0,\ldots,p_{k}\}$.
		\item[\it (vii)] $A_{k}^{m} = \sigma \langle \overset{\longrightarrow}{\rho} [b_{j}^{n_{m}}]; A_{j}^{n_{m}} ; [b_{k}^{m}, B_{k}^{m}] \rangle$ for every $m\in \{1,\ldots,p_{k}\}$, where $j\in\{0,\ldots, k-1\}$ and $n_{m} \in \{1,\ldots,p_{j}\}$ are such that $\mathtt{depth}_{A_{j}}(b_{j}^{n_{m}}) = n_{j}$ and $a \sqsubseteq b_{j}^{n_{m}} \sqsubset b_{k}^{m}$ with $|b_{j}^{n_{m}}| = |b_{k}^{m}|-1$, and being $\overset{\longrightarrow}{\rho} [b_{j}^{n_{m}}]$ the partial run in the game $\mathcal{G}_{[a,M]}(\mathcal{X})$ coded by $b_{j}^{n_{m}}$ in some previous stage of the construction, specifically in the $j$-th stage, for which $\sigma \langle \overset{\longrightarrow}{\rho} [b_{j}^{n_{m}}] \rangle = A_{j}^{n_{m}}$.
	\end{enumerate}
	
    \medskip
    
    By induction, suppose that the objects $C_{k}^{m-1} \geq B_{k}^{m-1} \geq A_{k}^{m-1}$ in $[a,A_{k}]$ have been defined for some $1\leq m\leq p_{k}$, where $C_{k}^{m-1} \in [a_{k}]$ and $B_{k}^{m-1}, A_{k}^{m-1} \in [b_{k}^{m-1}]$. Since $r_{n_{k}}(C_{k}^{m-1}) = r_{n_{k}}(A_{k}) = a_{k}$, then $\mathtt{depth}_{C_{k}^{m-1}}(b_{k}^{m-1}) = \mathtt{depth}_{A_{k}}(b_{k}^{m-1}) = n_{k}$, hence $\mathtt{depth}_{A_{k}^{m-1}}(b_{k}^{m-1}) \leq \mathtt{depth}_{C_{k}^{m-1}}(b_{k}^{m-1}) < \infty$; so, by applying \textbf{A.3.2} we infer that there is an object $C_{k}^{m}\in [a_{k}, C_{k}^{m-1}]$ such that $[b_{k}^{m-1},C_{k}^{m}] \subseteq [b_{k}^{m-1},A_{k}^{m-1}]$. Also, notice that $\mathtt{depth}_{C_{k}^{m}}(b_{k}^{m}) = \mathtt{depth}_{A_{k}}(b_{k}^{m}) = n_{k}$ because $C_{k}^{m}\in [a_{k}, A_{k}]$ with $a_{k} = r_{n_{k}}(A_{k})$; thus, by \textbf{A.3.1} we have that $\mathtt{depth}_{C_{k}^{m}}(b_{k}^{m})<\infty$ implies $[b_{k}^{m},C_{k}^{m}] \neq \emptyset$, then we take any object $B_{k}^{m} \in [b_{k}^{m},C_{k}^{m}]$.

    \medskip
	
    Furthermore, we know that $b_{k}^{m} \leq_{\text{fin}} a_{k}$; thus, if $d\in \mathcal{AR}$ is the approximation such that $a \sqsubseteq d \sqsubset b_{k}^{m}$ and $|d| = |b_{k}^{m}|-1$, then by \textbf{A.2.3} we deduce that there is an approximation $a_{j} \sqsubset a_{k}$ with $j\in\{0,\ldots,k-1\}$ such that $d \leq_{\text{fin}} a_{j}$; in fact, we can assume that $j$ is the smallest satisfying this property, in which case $\mathtt{depth}_{A_{j}}(d)=n_{j}$ because $a_{j} = r_{n_{j}}(A_{j})$, consequently there is some $n_{m}\in\{1,\ldots,p_{j}\}$ such that $d= b_{j}^{n_{m}}$, therefore $\mathtt{depth}_{A_{j}} (b_{j}^{n_{m}}) = n_{j}$ and $a \sqsubseteq b_{j}^{n_{m}} \sqsubset b_{k}^{m}$ with $|b_{j}^{n_{m}}| = |b_{k}^{m}|-1$. 

    \medskip
	
    Now, in the $j$-th stage of the previously constructed process, we consider the partial run $\overset{\longrightarrow}{\rho} [b_{j}^{n_{m}}]$ in the game $\mathcal{G}_{[a,M]}(\mathcal{X})$ coded by the approximation $b_{j}^{n_{m}}$, for which we have that $\sigma \langle \overset{\longrightarrow}{\rho} [b_{j}^{n_{m}}] \rangle = A_{j}^{n_{m}}$; thus, at some point, $A_{j}^{n_{m}}$ was the last move made by the player $\textrm{I}$ according to the winning strategy $\sigma$. To continue, notice that $B_{k}^{m} \in [b_{k}^{m}, A_{j}^{n_{m}}]$, since $B_{k}^{m}\in [b_{k}^{m}, C_{k}^{m}] \subseteq [b_{k}^{m}, A_{k}] \subseteq [b_{j}^{n_{m}}, C_{j}^{n_{m}+1}] \subseteq [b_{j}^{n_{m}}, A_{j}^{n_{m}}]$ whenever $1\leq n_{m}<p_{j}$, and also $B_{k}^{m}\in [b_{k}^{m}, C_{k}^{m}] \subseteq [b_{k}^{m}, A_{k}] \subseteq [a_{j}, A_{j+1}]$ if $n_{m}=p_{j}$. Therefore, we define the object $A_{k}^{m} \in [b_{k}^{m},B_{k}^{m}]$ given by $A_{k}^{m} = \sigma \langle \overset{\longrightarrow}{\rho} [b_{j}^{n_{m}}] ; A_{j}^{n_{m}} ; [b_{k}^{m},B_{k}^{m}] \rangle$, so that $A_{k}^{m}$ represents the next move of the player \textrm{I} in the game $\mathcal{G}_{[a,M]}(\mathcal{X})$ using the winning strategy $\sigma$, in response to the throw $[b_{k}^{m},B_{k}^{m}]$ executed by the player \textrm{II}. Consequently, in this way we have constructed the desired objects $C_{k}^{m} \geq B_{k}^{m} \geq A_{k}^{m}$ for all $m\leq p_{k}$.

    \medskip
	
    In the end, for the last case corresponding to $b_{k}^{p_{k}}=a_{k}$, we obtain that $A_{k}^{p_{k}} = \sigma \langle \overset{\longrightarrow}{\rho} [a_{k-1}] ; A_{k} ; [a_{k},B_{k}^{p_{k}}] \rangle$, then $A_{k}^{p_{k}} \in [a_{k},B_{k}^{p_{k}}]$ and hence $r_{n_{k}} (A_{k}^{p_{k}})=a_{k}$. Therefore, we finally define the object $A_{k+1}\in [a_{k},A_{k}]$ by setting $A_{k+1} = A_{k}^{p_{k}}$, and we take the approximation $a_{k+1}\in\mathcal{AR}$ given by $a_{k+1}=r_{n_{k+1}}(A_{k+1})$, so that $a_{k} \sqsubset a_{k+1}$ and $|a_{k+1}|=|a_{k}|+1$.
    
    \medskip
	
	In this way, we have successfully constructed, among other considerations, the fusion sequence $\{ [n_{k},A_{k}] \}_{k\in\mathbb{N}} \subseteq \wp([a,M])$, where $n_{k} = \mathtt{depth}_{M}(a)+k =|a|+k$ and $r_{n_{k}}(A_{k})=a_{k}$ for each $k\in\mathbb{N}$. Thus, if we define the object $H\in [a,M]$ as the limit $H= \lim_{n_{k}} A_{k}$ of the fusion sequence $\{ [n_{k},A_{k}] \}_{k\in\mathbb{N}}$, then $H\in [a_{k},A_{k}]$ for every $k\in \mathbb{N}$, and consequently $\{H\} = \bigcap_{k\in\mathbb{N}} [a_{k}]$.

    \medskip
	
	Next, we proceed to consider the following complete run in the game $\mathcal{G}_{[a,M]}(\mathcal{X})$:

    \medskip
	
	\begin{minipage}[c]{0.25\linewidth}
		\begin{center}
			$\mathcal{G}_{[a,M]}(\mathcal{X})$
		\end{center}
	\end{minipage}
	\begin{minipage}[c]{0.75\linewidth}
		\begin{equation*}
			\begin{matrix}
				\textrm{I} &  & A_{1} &  & A_{2} & & A_{3} & & \cdots & \\ 
				\textrm{II} &  &  & [ a_{1}, B_{1}^{p_{1}} ] &  & [ a_{2}, B_{2}^{p_{2}} ] & & [ a_{3}, B_{3}^{p_{3}} ] &  & \cdots
			\end{matrix}
		\end{equation*}
	\end{minipage}

    \medskip
	
	Therefore, this complete run is obtained by the player \textrm{I} applying the winning strategy $\sigma$ in the game $\mathcal{G}_{[a,M]}(\mathcal{X})$, since $A_{1} = \sigma \langle \emptyset \rangle$ and $A_{k+1} = \sigma \langle \overset{\longrightarrow}{\rho} [a_{k-1}] ; A_{k} ; [a_{k},B_{k}^{p_{k}}] \rangle$ for every $k\in \mathbb{N}-\{0\}$. As a result, we deduce that the player \textrm{I} wins the game $\mathcal{G}_{[a,M]}(\mathcal{X})$; thus, we conclude that $H\in \mathcal{X}$ because $\{H\} = \bigcap_{k\in\mathbb{N}} [a_{k}]$.

    \medskip
	
	Finally, let us verify that the object $H\in [a,M] \cap \mathcal{X}$ satisfies that $[a,H] \subseteq \mathcal{X}$. Indeed, if $B\in [a,H]$ then $B\leq H$ with $r_{|a|}(B) = a = r_{|a|}(H)$, thus, by virtue of \textbf{A.2.2}, we infer that for each $m\in \mathbb{N}$ there is some $j_{m} \in \mathbb{N}$ such that $r_{|a|+m}(B) \leq_{\text{fin}} r_{|a|+j_{m}} (H)$, and hence $r_{|a|+m}(B) \leq_{\text{fin}} a_{j_{m}}$; in addition, without loss of generality, we may assume that $j_{m}$ is the smallest satisfying this property, in which case $\mathtt{depth}_{A_{j_{m}}}(r_{|a|+j_{m}}(B)) = n_{j_{m}}$, consequently $r_{|a|+j_{m}}(B) = b_{j_{m}}^{i_{m}}$ for some $i_{m}\in \{1,\ldots,p_{j_{m}}\}$. Therefore, the object $B\in[a,H]$ satisfies that $\{B\} = \bigcap_{m\in\mathbb{N}} [b_{j_{m}}^{i_{m}}]$, and in view of the foregoing, we now proceed to consider the following complete run in the game $\mathcal{G}_{[a,M]}(\mathcal{X})$:

    \medskip
	
	\begin{minipage}[c]{0.25\linewidth}
		\begin{center}
			$\mathcal{G}_{[a,M]}(\mathcal{X})$
		\end{center}
	\end{minipage}
	\begin{minipage}[c]{0.75\linewidth}
		\begin{equation*}
			\begin{matrix}
				\textrm{I} &  & A_{1} &  & A_{j_{2}}^{i_{2}} & & A_{j_{3}}^{i_{3}} & & \cdots & \\ 
				\textrm{II} &  &  & [ b_{j_{1}}^{i_{1}} , B_{j_{1}}^{i_{1}} ] &  & [ b_{j_{2}}^{i_{2}} , B_{j_{2}}^{i_{2}} ] & & [ b_{j_{3}}^{i_{3}} , B_{j_{3}}^{i_{3}} ] &  & \cdots
			\end{matrix}
		\end{equation*}
	\end{minipage}

    \medskip
	
	It should be noted that this complete run is obtained by applying the winning strategy $\sigma$ of the player \textrm{I} in the game $\mathcal{G}_{[a,M]}(\mathcal{X})$, hence the player \textrm{I} wins the game $\mathcal{G}_{[a,M]}(\mathcal{X})$; as a result, we deduce that $B\in \mathcal{X}$ because $\{B\} = \bigcap_{m\in\mathbb{N}} [b_{j_{m}}^{i_{m}}]$. Consequently, we finally conclude that $[a,H] \subseteq \mathcal{X}$.
\end{proof}

\smallskip

Motivated by the concept of selective coideal on the Ellentuck space, given by Mathias in \cite{Mathias}, as well as by the abstract approach of local Ramsey theory in the context of axiomatized  topological Ramsey spaces, proposed by Di Prisco, Mijares, and Nieto in \cite{DiPrisco-Mijares-Nieto}, we now introduce the new broad class of selective axiomatized topological Ramsey spaces.

\smallskip

\begin{definition} 
	Let $(\mathcal{R},\leq,r)$ be an axiomatized topological Ramsey space. Given any $[a,B] \neq \emptyset$ and a sequence $\{D_{n}\}_{n\in \mathbb{N}} \subseteq [a,B]$, we say that an object $D_{\infty} \in [a,B]$ is a \textit{diagonalization} of $\{D_{n}\}_{n\in \mathbb{N}}$ whenever $\emptyset \neq [b,D_{\infty}] \subseteq [b,D_{n_{b}}]$ for every $b\in \mathcal{AR} \!\restriction\! [a,D_{\infty}]$, where $n_{b} =\mathtt{depth}_{B}(b) - \mathtt{depth}_{B}(a)$.
\end{definition}


\begin{definition}
	An axiomatized topological Ramsey space $(\mathcal{R},\leq,r)$ is said to be \textit{selective}, if for any $[a,B] \neq \emptyset$ and for all sequence $\{D_{n}\}_{n\in \mathbb{N}} \subseteq [a,B]$ such that $D_{n+1} \leq D_{n}$ for each $n\in\mathbb{N}$, there exists some $D_{\infty} \in [a,B]$ such that $D_{\infty}$ is a diagonalization of $\{D_{n}\}_{n\in \mathbb{N}}$.
\end{definition}

\smallskip

We point out that many well known examples of axiomatized topological Ramsey spaces are selective; however, there are exceptions, for instance, spaces  of strong subtrees (see \cite{Milliken3, Todorcevic(BookRamsey)}).

\medskip

We turn now to the case when player \textrm{II} has a winning strategy in the abstract Kastanas game. To do this, first we show a lemma about selectivity that we will need later, which is inspired by the development of a local version of Kastanas games given by Matet in \cite{Matet1}.

\smallskip

\begin{lemma} \label{game-lemma}
	Let $(\mathcal{R},\leq,r)$ be a selective axiomatized topological Ramsey space. Let $[a,B]\neq \emptyset$ with $a\sqsubset B$, and let $f:[a,B] \rightarrow \mathcal{AR}_{|a|+1}$ and $g:[a,B] \rightarrow [a,B]$ be functions such that $a\sqsubset f(A)$ and $g(A)\in [f(A),A]$ for each $A\in [a,B]$. Then, there exists some $E_{f,g}\in [a,B]$ which satisfies that for every $q\in r_{|a|+1} \text{''}\, [a,E_{f,g}]$ there is $A\in [a,B]$ such that $f(A)=q$ and $\emptyset \neq [q,E_{f,g}] \subseteq [q,g(A)]$.
\end{lemma}

\begin{proof}
	Let $(\mathcal{R},\leq,r)$ be a selective axiomatized topological Ramsey space. Given a basic open set $[a,B]\neq \emptyset$ with $a\sqsubset B$, we fix functions $f:[a,B] \rightarrow \mathcal{AR}_{|a|+1}$ and $g:[a,B] \rightarrow [a,B]$ such that $a\sqsubset f(A)$ and $g(A)\in [f(A),A]$ for each $A\in [a,B]$.

    \medskip
	
	First of all, for every $p\in r_{|a|+1} \text{''}\, [a,B]$ we consider the set $\mathcal{W}_{p} \subseteq [a,B]$ defined by 
	\begin{center}
			$\mathcal{W}_{p}= \{ D\in[a,B] \,|\, ( \exists\, A\in[a,B] ) ( g(A)=D \wedge f(A)=p ) \}$. 
	\end{center}
	We proceed by recursively constructing a sequence of objects $\{D_{n}\}_{n\in\mathbb{N}} \subseteq [a,B]$, with $D_{0}=B$ and $D_{n+1} \leq D_{n}$ for all $n\in \mathbb{N}$, as follows:

    \medskip
	
	By induction, let $n\in\mathbb{N}$ and suppose that the objects $D_{0} \geq D_{1} \geq \cdots \geq D_{n-1}$ in $[a,B]$ have been constructed. Now, let $\Gamma_{n} = \{ p\in r_{|a|+1} \text{''}\, [a,B] \,|\, \mathtt{depth}_{B}(p) = |a|+n \}$, then the set $\Gamma_{n}$ is finite by virtue of \textbf{A.2.1}, since $\Gamma_{n} \subseteq \{ p\in\mathcal{AR} \,|\, p \leq_{\text{fin}} r_{|a|+n}(B) \}$; so, we fix any enumeration $\Gamma_{n} = \{ p_{n}^{1}, \ldots, p_{n}^{k_{n}} \}$ with $k_{n}\in \mathbb{N}$. Next, we recursively define a finite sequence of objects $D_{n}^{0}, D_{n}^{1}, \ldots, D_{n}^{k_{n}}$ in $[a,D_{n-1}]$. To do this, we put $D_{n}^{0} = D_{n-1}$, and assuming that the object $D_{n}^{i-1}$ has been determined for $1\leq i \leq k_{n}$, we consider the set $\mathcal{W}_{p_{n}^{i}} \cap [a,D_{n}^{i-1}]$ and analyze the following cases: on the one hand, if $\mathcal{W}_{p_{n}^{i}} \cap [a,D_{n}^{i-1}]=\emptyset$, then we take $D_{n}^{i}=D_{n}^{i-1}$; on the other hand, if $\mathcal{W}_{p_{n}^{i}} \cap [a,D_{n}^{i-1}] \neq \emptyset$, then we take $D_{n}^{i}$ as any fixed object $D_{n}^{i} \in \mathcal{W}_{p_{n}^{i}} \cap [a,D_{n}^{i-1}]$. In this way, we get finitely many objects $D_{n}^{0} \geq D_{n}^{1} \geq \cdots \geq D_{n}^{k_{n}}$ in $[a,D_{n-1}]$; thus, we define $D_{n} \in[a,D_{n-1}]$ as the object $D_{n}= D_{n}^{k_{n}}$.

    \medskip
	
	As an additional remark to the construction just explained, it should be noted that the family of approximations $r_{|a|+1} \text{''}\, [a,B]$ can be described in terms of the sequence of finite sets $\{ \Gamma_{n} \}_{n\in\mathbb{N}}$ as follows:   
	\begin{center}
		$r_{|a|+1} \text{''}\, [a,B] = \bigcup_{n\in\mathbb{N}} \Gamma_{n} = \{ p_{n}^{i} \,|\, n\in\mathbb{N} \wedge 1\leq i\leq k_{n} \}$.
	\end{center}

    \medskip
    
	Now, we know that the sequence of objects $\{D_{n}\}_{n\in\mathbb{N}} \subseteq [a,B]$ is decreasing with respect to $\leq$, then, using the fact that the axiomatized topological Ramsey space $(\mathcal{R},\leq,r)$ is selective, we get an object $D_{\infty} \in [a,B]$ that corresponds to a diagonalization of the sequence $\{D_{n}\}_{n\in\mathbb{N}}$; therefore, for every $b\in \mathcal{AR} \!\restriction\! [a,D_{\infty}]$, if $\mathtt{depth}_{B}(b) = |a|+n$ then $\emptyset \neq [b,D_{\infty}] \subseteq [b,D_{n}]$.

    \medskip
	
	Next, let $\mathcal{O} \subseteq \mathcal{AR}_{|a|+1}$ be the collection of approximations defined by
	\begin{center}
		$\mathcal{O} = \{ p_{n}^{i} \in r_{|a|+1} \text{''}\, [a,B] \;|\; \mathcal{W}_{p_{n}^{i}} \cap [a,D_{n}^{i-1}] \neq \emptyset \}.$
	\end{center} 
	Applying the abstract pigeonhole principle \textbf{A.4} to the family $\mathcal{O}$, we deduce that there exists an object $E_{f,g} \in [a,D_{\infty}] \subseteq [a,B]$ such that $r_{|a|+1} \text{''}\, [a,E_{f,g}]$ is either contained in $\mathcal{O}$ or disjoint from $\mathcal{O}$, so
	\begin{center}
		$r_{|a|+1} \text{''}\, [a,E_{f,g}] \subseteq \mathcal{O} \;\;\;\;\; \vee \;\;\;\;\; \mathcal{O} \cap r_{|a|+1} \text{''}\, [a,E_{f,g}] =\emptyset$. 
	\end{center} 

    \medskip
    
	Nevertheless, we claim that the second option is not possible, since  $f(E_{f,g}) \in \mathcal{O} \cap r_{|a|+1} \text{''}\, [a,E_{f,g}]$. Indeed, on the one hand, $a \sqsubset f(E_{f,g})$ with $|f(E_{f,g})|=|a|+1$ and also $g(E_{f,g}) \in [ f(E_{f,g}), E_{f,g} ]$, so that $r_{|a|+1} (g(E_{f,g})) = f(E_{f,g})$, and hence $f(E_{f,g}) \in r_{|a|+1} \text{''}\, [a,E_{f,g}]$. On the other hand, $f(E_{f,g}) = p_{n}^{i}$ for some $n\in\mathbb{N}$ and $1\leq i\leq k_{n}$, so that $\mathtt{depth}_{B} (f(E_{f,g})) =|a|+n$ and also $g(E_{f,g}) \in \mathcal{W}_{p_{n}^{i}}$, then $[ f(E_{f,g}),E_{f,g} ] \subseteq [ f(E_{f,g}),D_{\infty} ] \subseteq [ f(E_{f,g}),D_{n} ] \subseteq [a,D_{n}] \subseteq [a,D_{n}^{i-1}]$, thus $g(E_{f,g}) \in [ a,D_{n}^{i-1} ]$; as a result, we deduce that $\mathcal{W}_{f(E_{f,g})} \cap [a,D_{n}^{i-1}] = \mathcal{W}_{p_{n}^{i}} \cap [a,D_{n}^{i-1}] \neq \emptyset$, and hence $f(E_{f,g}) \in \mathcal{O}$. Consequently, we must have that $r_{|a|+1} \text{''}\, [a,E_{f,g}] \subseteq \mathcal{O}$.

    \medskip
	
	Finally, let us verify that for every $q\in r_{|a|+1} \text{''}\, [a,E_{f,g}]$ there is some $A\in [a,B]$ such that $f(A)=q$ and $\emptyset \neq [q,E_{f,g}] \subseteq [q,g(A)]$. Indeed, if $q\in r_{|a|+1} \text{''}\, [a,E_{f,g}]$ then $q=p_{n}^{i}$ for some $n\in\mathbb{N}$ and $1\leq i\leq k_{n}$, and since $r_{|a|+1} \text{''}\, [a,E_{f,g}] \subseteq \mathcal{O}$, it follows that $q\in \mathcal{O}$. Thus, we have that $\mathcal{W}_{p_{n}^{i}} \cap [a,D_{n}^{i-1}] \neq \emptyset$, and in particular, this implies that $D_{n}^{i} \in \mathcal{W}_{p_{n}^{i}}$, that is $D_{n}^{i} \in \mathcal{W}_{q}$. Then, there exists some $A \in [a,B]$ such that $f(A)=q$ and $g(A)=D_{n}^{i}$, so that $[q,g(A)] = [q,D_{n}^{i}]$. Moreover, since $\mathtt{depth}_{B} (q) =|a|+n$, we deduce that $[q,E_{f,g}] \subseteq [q,D_{\infty}] \subseteq [q,D_{n}] \subseteq [q,D_{n}^{i}]$, which is why we conclude that $\emptyset \neq [q,E_{f,g}] \subseteq [q,g(A)]$.	
\end{proof}

\smallskip

Applying the above lemma, we now consider selective axiomatized topological Ramsey spaces to derive the following key result that characterizes when player \textrm{II} has a winning strategy in the abstract Kastanas game.

\smallskip

\begin{proposition} \label{strategy-II}
	Let $(\mathcal{R},\leq,r)$ be a selective axiomatized topological Ramsey space, and let $\mathcal{X} \subseteq \mathcal{R}$ and $[a,M] \neq \emptyset$ be given. Then, the player \textrm{II} has a winning strategy in the game $\mathcal{G}_{[a,M]}(\mathcal{X})$ if and only if for every $N\in[a,M]$ there is $H\in [a,N]$ such that $[a,H] \cap \mathcal{X} = \emptyset$.
\end{proposition}

\begin{proof}
	Given a selective axiomatized topological Ramsey space $(\mathcal{R},\leq,r)$, we fix a set of objects $\mathcal{X} \subseteq \mathcal{R}$ and a basic open set $[a,M] \neq \emptyset$, for which we assume without loss of generality that $a \sqsubset M$. So, we proceed to consider the abstract Kastanas game $\mathcal{G}_{[a,M]}(\mathcal{X})$. 

    \medskip

	First of all, suppose that for all $N\in [a,M]$ there exists an object $H\in [a,N]$ such that $[a,H] \cap \mathcal{X} = \emptyset$. In this case, a winning strategy for the player \textrm{II} in the game $\mathcal{G}_{[a,M]}(\mathcal{X})$ is given as follows: the player \textrm{II} must select the basic open set $[r_{|a|+1}(H), H]$ as the first move, in response to the first throw $N$ executed by the player \textrm{I}. Indeed, consider any complete run in the game $\mathcal{G}_{[a,M]}(\mathcal{X})$ of the form 

    \medskip
	
	\begin{minipage}[c]{0.23\linewidth}
		\begin{center}
			$\mathcal{G}_{[a,M]}(\mathcal{X})$
		\end{center}
	\end{minipage}
	\begin{minipage}[c]{0.75\linewidth}
		\begin{equation*}
			\begin{matrix}
				\textrm{I} &  & N &  & A_{2} & & A_{3} & & \cdots & \\ 
				\textrm{II} &  &  & [ r_{|a|+1}(H), H ] &  & [ a_{2}, B_{2} ] & & [ a_{3}, B_{3} ] &  & \cdots
			\end{matrix}
		\end{equation*}
	\end{minipage}

    \medskip
	
	So, notice that the object $B_{\infty}\in [a,M]$ given by $\{B_{\infty}\} = \bigcap_{k=2}^{\infty} [a_{k}]$ is such that $B_{\infty}\in [a,H]$, thus $B_{\infty} \notin \mathcal{X}$ and consequently the player \textrm{II} wins the game $\mathcal{G}_{[a,M]}(\mathcal{X})$. 

    \medskip
	
	Conversely, suppose that the player \textrm{II} has a winning strategy $\tau$ in the game $\mathcal{G}_{[a,M]}(\mathcal{X})$, and let $N \in [a,M]$ be a fixed but arbitrary object. Now, our goal is to construct a winning strategy $\sigma$ for the player \textrm{I} in the game $\mathcal{G}_{[a,N]}(\mathcal{R}\setminus \mathcal{X})$. 

    \medskip

	In order to successfully accomplish our goal, we must iteratively apply lemma \ref{game-lemma} to recursively construct a complete run in the game $G_{[a,N]}(\mathcal{R}\setminus\mathcal{X})$, with consecutive moves of the player \textrm{II} given by a sequence of basic open sets $\{[a_{n},B_{n}]\}_{n=1}^{\infty}$, for which we define sequences of functions $\{f_{n}\}_{n=1}^{\infty}$ and $\{g_{n}\}_{n=1}^{\infty}$ as well as sequences of objects $\{A_{n}\}_{n=1}^{\infty}$ and $\{E_{f_{n},g_{n}}\}_{n=1}^{\infty}$ in the space $\mathcal{R}$, such that:
		
	\setlist{nolistsep}
	\begin{enumerate}
		\setlength{\itemsep}{0pt}	
		\item[(1)] $[f_{1}(A),g_{1}(A)] = \tau \langle A \rangle$ for every $A\in [a,N]$.
		\item[(2)] $[f_{n+1}(A),g_{n+1}(A)] = \tau \left\langle A_{1}; [f_{1}(A_{1}),g_{1}(A_{1})]; \ldots; A_{n}; [f_{n}(A_{n}),g_{n}(A_{n})]; A  \right\rangle$ for every $A\in [a_{n},B_{n}]$.
		\item[(3)] $A_{1}, E_{f_{1},g_{1}} \in [a,N]$ and $A_{n+1}, E_{f_{n+1},g_{n+1}} \in [a_{n},B_{n}]$.
		\item[(4)] $f_{n}(A_{n}) = a_{n}$ and $B_{n} \in [a_{n}, E_{f_{n},g_{n}}] \subseteq [a_{n}, g_{n}(A_{n})]$. 
	\end{enumerate}

    \medskip
    
    For the first stage of the process, we take functions $f_{1}: [a,N] \rightarrow \mathcal{AR}_{|a|+1}$ and $g_{1}: [a,N] \rightarrow [a,N]$ defined by $[f_{1}(A),g_{1}(A)] = \tau \langle A \rangle$ for each $A\in [a,N]$, so the basic open set $[f_{1}(A),g_{1}(A)]$ is the first move of the player \textrm{II} using the winning strategy $\tau$ in the game $\mathcal{G}_{[a,M]}(\mathcal{X})$, in response to any throw $A$ in $[a,N]$ given by the player \textrm{I}; then, we have $a \sqsubset f_{1}(A)$ and $g_{1}(A) \in [f_{1}(A),A]$ for each $A\in [a,N]$. Therefore, by lemma \ref{game-lemma}, we infer that there exists an object $E_{f_{1},g_{1}} \in [a,N]$ with the property that for every $q\in r_{|a|+1} \text{''}\, [a,E_{f_{1},g_{1}}]$ there is $A\in [a,N]$ such that $f_{1}(A)=q$ and $\emptyset \neq [q,E_{f_{1},g_{1}}] \subseteq [q,g_{1}(A)]$. Thus, we define $\sigma \langle \emptyset \rangle = E_{f_{1},g_{1}}$.

    \medskip
	
	Again, we give a detailed description of the next stage of the construction for the benefit of the reader. In the next stage of the process, suppose that the short partial run given by $E_{f_{1},g_{1}}$ followed by $[a_{1},B_{1}]$, as the first moves of the players \textrm{I} and \textrm{II}, respectively, has been developed in the game $G_{[a,N]}(\mathcal{R}\setminus\mathcal{X})$, so that $B_{1} \in [a_{1},E_{f_{1},g_{1}}]$ and $a \sqsubset a_{1}$ with $|a_{1}| = |a|+1$. Since $a_{1} \in r_{|a|+1} \text{''}\, [a,E_{f_{1},g_{1}}]$, then there is $A_{1}\in [a,N]$ such that $f_{1}(A_{1})=a_{1}$ and $\emptyset \neq [a_{1},E_{f_{1},g_{1}}] \subseteq [a_{1},g_{1}(A_{1})]$, so $B_{1} \in [a_{1},g_{1}(A_{1})]$ and hence $\emptyset \neq [a_{1},B_{1}] \subseteq [a_{1},g_{1}(A_{1})]$. Let $f_{2}: [a_{1},B_{1}] \rightarrow \mathcal{AR}_{|a_{1}|+1}$ and $g_{2}: [a_{1},B_{1}] \rightarrow [a_{1},B_{1}]$ be functions defined by $[f_{2}(A),g_{2}(A)] = \tau \langle A_{1} ; [a_{1},g_{1}(A_{1})] ; A \rangle$ for each $A\in [a_{1},B_{1}]$, so the second move made by the player \textrm{II} using the winning strategy $\tau$ in the game $\mathcal{G}_{[a,M]}(\mathcal{X})$ is $[f_{2}(A),g_{2}(A)]$, in response to any throw $A$ in $[a_{1},B_{1}]$ given by the player \textrm{I}, who had selected $A_{1}$ as previous throw; then, $a_{1} \sqsubset f_{2}(A)$ and $g_{2}(A) \in [f_{2}(A),A]$ for each $A\in [a_{1},B_{1}]$. Therefore, by lemma \ref{game-lemma}, there exists $E_{f_{2},g_{2}} \in [a_{1},B_{1}]$ with the property that for every $q\in r_{|a_{1}|+1} \text{''}\, [a_{1},E_{f_{2},g_{2}}]$ there is $A\in [a_{1},B_{1}]$ such that $f_{2}(A)=q$ and $\emptyset \neq [q,E_{f_{2},g_{2}}] \subseteq [q,g_{2}(A)]$. Thus, we define $\sigma \langle E_{f_{1},g_{1}}; [a_{1},B_{1}] \rangle = E_{f_{2},g_{2}}$.

    \medskip
	
	In this order of ideas, by induction we fix $n\in \mathbb{N}-\{0\}$ and suppose that the $j$-th stage of the process has been constructed for each $j\in\{1,\ldots,n\}$, that is, we are assuming defined the first $n$ stages of the strategy $\sigma$ for the player \textrm{I} in the game $G_{[a,N]}(\mathcal{R}\setminus\mathcal{X})$, where the finite sequence of basic open sets $\{[a_{j},B_{j}]\}_{j=1}^{n}$ corresponds to the first $n$ consecutive moves of the player \textrm{II}, which has associated both functions $\{f_{n}\}_{j=1}^{n}$ and $\{g_{n}\}_{j=1}^{n}$ as well as objects $\{A_{j}\}_{j=1}^{n}$ and $\{E_{f_{j},g_{j}}\}_{j=1}^{n}$ in the space $\mathcal{R}$, satisfying the conditions established in items (1), (2), (3), (4) above, such that $\sigma \langle \emptyset \rangle = E_{f_{1},g_{1}}$ and $\sigma \langle E_{f_{1},g_{1}} ; [a_{1},B_{1}] ; \ldots ; E_{f_{j-1},g_{j-1}} ; [a_{j-1},B_{j-1}] \rangle = E_{f_{j},g_{j}}$ for every $1\leq j\leq n$. 

    \medskip
	
	Now, in the game $G_{[a,N]}(\mathcal{R}\setminus\mathcal{X})$, we consider the partial run consisting of the consecutive moves $E_{f_{1},g_{1}}, \ldots, E_{f_{n},g_{n}}$ executed by the player \textrm{I},  interleaved with the successive throws $[a_{1},B_{1}], \ldots, [a_{n},B_{n}]$ made by the player \textrm{II}; so, in particular we have that $B_{n}\in [a_{n},E_{f_{n},g_{n}}]$ and $a_{n-1} \sqsubset a_{n}$ with $|a_{n}|=|a_{n-1}|+1$, and hence $a_{n} \in r_{|a_{n-1}|+1} \text{''}\, [a_{n-1},E_{f_{n},g_{n}}]$. Then, there exists some object $A_{n}\in [a_{n-1},B_{n-1}]$ such that $f_{n}(A_{n}) = a_{n}$ and $\emptyset \neq [a_{n},E_{f_{n},g_{n}}] \subseteq [a_{n},g_{n}(A_{n})]$, this by virtue of the fact that the object $E_{f_{n},g_{n}}$ was previously constructed using lemma \ref{game-lemma}. Therefore, we deduce that $B_{n} \in [a_{n},g_{n}(A_{n})]$, and hence $\emptyset \neq [a_{n},B_{n}] \subseteq [a_{n},g_{n}(A_{n})]$.

    \medskip
	
	Next, we consider functions $f_{n+1}: [a_{n},B_{n}] \rightarrow \mathcal{AR}_{|a_{n}|+1}$ and $g_{n+1}: [a_{n},B_{n}] \rightarrow [a_{n},B_{n}]$ defined by $[f_{n+1}(A),g_{n+1}(A)] = \tau \langle A_{1} ; [a_{1},g_{1}(A_{1})] ; \ldots; A_{n} ; [a_{n},g_{n}(A_{n})] ; A \rangle$ for each $A\in [a_{n},B_{n}]$, that is, the basic open set $[f_{n+1}(A),g_{n+1}(A)]$ corresponds to the $(n+1)$-th move made by the player \textrm{II} using the winning strategy $\tau$ in the game $\mathcal{G}_{[a,M]}(\mathcal{X})$, in response to any throw $A$ in $[a_{n},B_{n}]$ given by the player \textrm{I}, who had successively selected $A_{1}, \ldots, A_{n}$ as previous throws; then, it follows that $a_{n} \sqsubset f_{n+1}(A)$ and $g_{n+1}(A) \in [f_{n+1}(A),A]$ for each $A\in [a_{n},B_{n}]$. Therefore, by applying lemma \ref{game-lemma}, we deduce the existence of some object $E_{f_{n+1},g_{n+1}} \in [a_{n},B_{n}]$ with the special property that for every $q\in r_{|a_{n}|+1} \text{''}\, [a_{n},E_{f_{n+1},g_{n+1}}]$ there is $A\in [a_{n},B_{n}]$ such that $f_{n+1}(A)=q$ and $\emptyset \neq [q,E_{f_{n+1},g_{n+1}}] \subseteq [q,g_{n+1}(A)]$. In view of the above, we define $\sigma \langle E_{f_{1},g_{1}} ; [a_{1},B{1}] ; \ldots ; E_{f_{n},g_{n}} ; [a_{n},B_{n}] \rangle = E_{f_{n+1},g_{n+1}}$.

    \medskip
	
	In this way, we have successfully constructed, among other considerations, the strategy $\sigma$ for the player \textrm{I} in the game $G_{[a,N]}(\mathcal{R} \setminus \mathcal{X})$, which consists of consecutively choosing the objects of the sequence $\{E_{f_{n},g_{n}}\}_{n=1}^{\infty}$ in each of the moves. Thus, the strategy $\sigma$ is defined as follows:
	\begin{equation*}
		\sigma \langle \emptyset \rangle = E_{f_{1},g_{1}} \,\text{ and }\, \sigma \left\langle E_{f_{1},g_{1}}; [a_{1},B_{1}]; \ldots; E_{f_{n},g_{n}}; [a_{n},B_{n}] \right\rangle = E_{f_{n+1},g_{n+1}} \,\text{ for every } n\in\mathbb{N}-\{0\}.
	\end{equation*}

    \medskip
	
	Now, let $B_{\infty} \in [a,N]$ be the object determined by $\{ B_{\infty} \} = \bigcap_{n=1}^{\infty} [a_{n}]$. To continue, we proceed to consider the following complete run in the game $\mathcal{G}_{[a,M]}(\mathcal{X})$:

    \medskip
	
	\begin{minipage}[c]{0.25\linewidth}
		\begin{center}
			$\mathcal{G}_{[a,M]}(\mathcal{X})$
		\end{center}
	\end{minipage}
	\begin{minipage}[c]{0.75\linewidth}
		\begin{equation*}
			\begin{matrix}
				\textrm{I} &  & A_{1} &  & A_{2} & & A_{3} & & \cdots & \\ 
				\textrm{II} &  &  & [ a_{1}, g_{1}(A_{1}) ] &  & [ a_{2}, g_{2}(A_{2}) ] & & [ a_{3}, g_{3}(A_{3}) ] &  & \cdots
			\end{matrix}
		\end{equation*}
	\end{minipage}

    \medskip
	
	It is worth noting that this complete run is obtained by applying the winning strategy $\tau$ for the player \textrm{II} in the game $\mathcal{G}_{[a,M]}(\mathcal{X})$, since $\tau \left\langle A_{1}; [a_{1},g_{1}(A_{1})]; \ldots; A_{n}; [a_{n},g_{n}(A_{n})]; A_{n+1} \right\rangle =[a_{n+1},g_{n+1}(A_{n+1})]$ for each $n\in \mathbb{N}-\{0\}$. As a result, we deduce that the player \textrm{II} wins the game $\mathcal{G}_{[a,M]}(\mathcal{X})$, thus $B_{\infty}\notin \mathcal{X}$ and hence $B_{\infty}\in \mathcal{R} \setminus \mathcal{X}$. 

    \medskip
	
	On the other hand, by applying the strategy $\sigma$ that has been recursively constructed for the player \textrm{I} in the game $G_{[a,N]}(\mathcal{R} \setminus \mathcal{X})$, where $\{E_{f_{n},g_{n}}\}_{n=1}^{\infty}$ corresponds to the sequence of consecutive moves of the player \textrm{I}, with $\{[a_{n},B_{n}]\}_{n=1}^{\infty}$ being the respective sequence of successive moves of the player \textrm{II}, then we obtain the following complete run in the game $G_{[a,N]}(\mathcal{R} \setminus \mathcal{X})$: 

    \medskip
	
	\begin{minipage}[c]{0.25\linewidth}
		\begin{center}
			$G_{[a,N]}(\mathcal{R} \setminus \mathcal{X})$
		\end{center}
	\end{minipage}
	\begin{minipage}[c]{0.75\linewidth}
		\begin{equation*}
			\begin{matrix}
				\textrm{I} &  & E_{f_{1},g_{1}} &  & E_{f_{2},g_{2}} & & E_{f_{3},g_{3}} & & \cdots & \\ 
				\textrm{II} &  &  & [ a_{1}, B_{1} ] &  & [ a_{2}, B_{2} ] & & [ a_{3}, B_{3} ] &  & \cdots
			\end{matrix}
		\end{equation*}
	\end{minipage}

    \medskip
	
	Consequently, by virtue of the fact that $B_{\infty} \in \mathcal{R} \setminus \mathcal{X}$, it follows that the player \textrm{I} wins the game $G_{[a,N]}(\mathcal{R} \setminus \mathcal{X})$; therefore, we conclude that the strategy $\sigma$ is actually a winning strategy for the player \textrm{I} in the game $G_{[a,N]}(\mathcal{R} \setminus \mathcal{X})$.

    \medskip
	
	Finally, since the player \textrm{I} has a winning strategy in the game $G_{[a,N]}(\mathcal{R} \setminus \mathcal{X})$, then by proposition \ref{strategy-I} we deduce that there exists some object $H\in [a,N]$ such that $[a,H] \subseteq \mathcal{R} \setminus \mathcal{X}$, and thus $[a,H] \cap \mathcal{X} = \emptyset$.
\end{proof}

\smallskip

We are finally ready to present the main result of our article, which corresponds to the generalization of theorem \ref{KastanasTheorem} for the abstract context. Thus, our abstract version of Kastanas theorem characterizes the Ramsey property through the existence of winning strategies in the abstract Kastanas game on the combinatorial structure of selective axiomatized topological Ramsey spaces. 

\smallskip

\begin{theorem} \label{Abstract-Kastanas-Theorem}
	[\textit{Abstract Kastanas Theorem}]. Let $(\mathcal{R},\leq,r)$ be a selective axiomatized topological Ramsey space. Then, for every $\mathcal{X} \subseteq \mathcal{R}$, the following statements are equivalent: 
	\setlist{nolistsep}
	\begin{enumerate}
		\setlength{\itemsep}{0pt}
		\item[(a)] $\mathcal{X}$ is Ramsey.
		\item[(b)] For all $[a,A]\neq \emptyset$, the abstract Kastanas game $\mathcal{G}_{[a,A]}(\mathcal{X})$ is determined.
	\end{enumerate}
\end{theorem}

\begin{proof}
	Let $(\mathcal{R},\leq,r)$ be a selective axiomatized topological Ramsey space, and fix both a set of objects $\mathcal{X} \subseteq \mathcal{R}$ and a basic open set $[a,A] \neq \emptyset$. 

    \medskip
	
	If the game $\mathcal{G}_{[a,A]}(\mathcal{X})$ is determined, then both proposition \ref{strategy-I} and proposition \ref{strategy-II} clearly imply that there exists some $H\in [a,A]$ such that either $[a,H] \subseteq \mathcal{X}$ or $[a,H] \cap \mathcal{X} =\emptyset$, so that $\mathcal{X}$ is a Ramsey set. 

    \medskip
	
	Conversely, suppose that the set $\mathcal{X}$ is Ramsey and that the player \textrm{II} does not have a winning strategy in the game $\mathcal{G}_{[a,A]}(\mathcal{X})$; thus, by proposition \ref{strategy-II}, there exists some $B\in [a,A]$ such that $[a,H] \cap \mathcal{X} \neq \emptyset$ for all $H\in[a,B]$. Since $\mathcal{X}$ is Ramsey, then for some $H\in [a,B] \subseteq [a,A]$ we have $[a,H] \subseteq \mathcal{X}$, so that the player \textrm{I} has a winning strategy in the game $\mathcal{G}_{[a,A]}(\mathcal{X})$, by virtue of proposition \ref{strategy-I}. Therefore, we conclude that the game $\mathcal{G}_{[a,A]}(\mathcal{X})$ is determined. 
\end{proof}

\smallskip

As a final comment, we mention that in view of these results it might be interesting to explore the Ramsey property on topological Ramsey spaces under determinacy axioms.

\section*{Acknowledgments}

The first author thanks the Fields Institute of the University of Toronto for its hospitality and partial support during the preparation of this article. 

\smallskip

The second author thanks the Mathematics Institute of the University of Barcelona for its hospitality and partial support during the preparation of this article.

\smallskip

The authors thank Natasha Dobrinen and Jos\'e Mijares for several invaluable conversations on the topics and results of this article.

\kern1em
\Addresses


\begin{thebibliography}{99}

\setlength{\itemsep}{0pt}

\bibitem{Carlson1} Carlson, T. (1987). \textit{An infinitary version of Graham-Leeb-Rothschild theorem}. Journal of Combinatorial Theory, Serie A. Vol. 44, pp. 22$-$33.

\bibitem{Carlson2} Carlson, T. (1988). \textit{Some unifying principles in Ramsey theory}. Discrete Mathematics. Vol. 68, pp. 117$-$169. 

\bibitem{Carlson-Simpson1} Carlson, T. $\&$ Simpson, S. (1984). \textit{A dual form of Ramsey's theorem}. Advances in Mathematics. Vol. 53, pp. 265$-$290.

\bibitem{Carlson-Simpson2} Carlson, T. $\&$ Simpson, S. (1990). \textit{Topological Ramsey theory}. In: `Mathematics of Ramsey Theory' (Eds: Nesetril, J. $\&$ R\"odl, V.). Berlin: Springer-Verlag. pp. 172$-$183.

\bibitem{DiPrisco-Mijares-Nieto} Di Prisco, C.A.; Mijares, J.G $\&$ Nieto, J.E. (2017). \textit{Local Ramsey theory: an abstract approach}. Mathematical Logic Quarterly. Vol. 63(5), pp. 384$-$396.

\bibitem{DiPrisco-Mijares-Uzcategui} Di Prisco, C.A.; Mijares, J.G. $\&$ Uzc\'{a}tegui, C.E. (2012). \textit{Ideal games and Ramsey sets}. Proceedings of the American Mathematical Society. Vol. 140(7), pp. 2255$-$2265.

\bibitem{Dobrinen1} Dobrinen, N. (2016). \textit{High dimensional Ellentuck spaces and initial chains in the Tukey structure of non-p-points}. Journal of Symbolic Logic. Vol. 81(1), pp. 237$-$263.

\bibitem{Dobrinen2} Dobrinen, N. (2016). \textit{Infinite-dimensional Ellentuck spaces and Ramsey-classification theorems}. Journal of Mathematical Logic. Vol. 16(1), pp. 1650003.

\bibitem{Dobrinen-Todorcevic1} Dobrinen, N. $\&$ Todor\v{c}evi\'{c}, S. (2014). \textit{A new class of Ramsey-classification theorems and their applications in the Tukey theory of ultrafilters, Part 1}. Transactions of the American Mathematical Society. Vol. 366(3), pp. 1659$-$1684. 

\bibitem{Dobrinen-Todorcevic2} Dobrinen, N. $\&$ Todor\v{c}evi\'{c}, S. (2015). \textit{A new class of Ramsey-classification theorems and their applications in the Tukey theory of ultrafilters, Part 2}. Transactions of the American Mathematical Society. Vol. 367(7), pp. 4627$-$4659. 

\bibitem{Dobrinen-Mijares-Trujillo} Dobrinen, N.; Mijares, J.G. $\&$ Trujillo, T. (2017). \textit{Topological Ramsey spaces from Fra\"\i ss\'e classes, Ramsey-classification theorems, and initial structures	in the Tukey types of p-points}. Archive for Mathematical Logic. Vol. 57, pp. 733$-$782.  

\bibitem{Dobrinen-Zucker} Dobrinen, N. $\&$ Zucker, A. (submitted). \textit{Infinite-dimensional Ramsey theory for binary free amalgamation classes}. arXiv:2303.04246 

\bibitem{Ellentuck} Ellentuck, E. (1974). \textit{A new proof that analytic sets are Ramsey}. Journal of Symbolic Logic. Vol. 39(1), pp. 163$-$165.

\bibitem{Erdos-Rado} Erd\H{o}s, P. $\&$ Rado, R. (1952). \textit{Combinatorial theorems on classifications of subsets of a given set}. Proceedings of the London Mathematical Society. Vol. 2(3), pp. 417$-$439.

\bibitem{Farah} Farah, I. (1998). \textit{Semiselective coideals}. Mathematika. Vol. 45(1), pp. 79$-$103.

\bibitem{Galvin-Prikry} Galvin, F. $\&$ Prikry, K. (1973). \textit{Borel sets and Ramsey's theorem}. Journal of Symbolic Logic. Vol. 38(2), pp. 193$-$198.

\bibitem{Gowers1} Gowers, W.T. (1992). \textit{Lipschitz functions on classical spaces}. European Journal of Combinatorics. Vol. 13, pp. 141$-$151.

\bibitem{Halbeisen} Halbeisen, L. (2017). \textit{Combinatorial set theory}. (2nd. ed.). London: Springer-Verlag.

\bibitem{Jech} Jech, T. (2003). \textit{Set theory: the third millennium edition}. Berlin: Springer.

\bibitem{Kastanas} Kastanas, I. (1983). \textit{On the Ramsey property for sets of reals}. Journal of Symbolic Logic. Vol. 48(4), pp. 1035$-$1045.

\bibitem{Kawach-Todorcevic} Kawach, J. $\&$ Todor\v{c}evi\'{c}, S. (2022). \textit{Topological Ramsey spaces of equivalence relations and a dual Ramsey theorem for countable ordinals}. Advances in Mathematics. Vol. 396, pp. 108194.

\bibitem{Matet1} Matet, P. (1993). \textit{Happy families and completely Ramsey sets}. Archive for Mathematical Logic. Vol. 32, pp. 151$-$171.

\bibitem{Matet2} Matet, P. (2000). \textit{A short proof of Ellentuck's theorem}. Proceedings of the American Mathematical Society. Vol. 129(4), pp. 1195$-$1197.

\bibitem{Mathias} Mathias, A.R.D. (1977). \textit{Happy families}. Annals of Mathematical Logic. Vol. 12(1), pp. 59$-$111.

\bibitem{Milliken1} Milliken, K. (1975). \textit{Ramsey's theorem with sums or unions}. Journal of Combinatorial Theory, Series A. Vol. 18, pp. 276$-$290.

\bibitem{Milliken3} Milliken, K. (1981). \textit{A partition theorem for the infinite subtrees of a tree}. Transactions of the American Mathematical Society. Vol. 263(1), pp. 137$-$148.

\bibitem{NashWilliams} Nash-Williams, C. (1965). \textit{On well-quasi-ordering transfinite sequences}. Mathematical Proceedings of the Cambridge Philosophical Society. Vol. 61, pp. 33$-$39.

\bibitem{Ramsey} Ramsey, F.P. (1930). \textit{On a problem of formal logic}. Proceedings of the London Mathematical Society. Vol. 30(4), pp. 264$-$286.

\bibitem{Silver} Silver, J. (1970). \textit{Every analytic set is Ramsey}. Journal of Symbolic Logic. Vol. 35(1), pp. 60$-$64.

\bibitem{Todorcevic(BookTopology)} Todor\v{c}evi\'{c}, S. (1997). \textit{Topics in topology}. Berlin: Springer-Verlag.

\bibitem{Todorcevic(BookRamsey)} Todor\v{c}evi\'{c}, S. (2010). \textit{Introduction to Ramsey spaces}. New Jersey: Princeton University Press.

\end{thebibliography}
\end{document}